\documentclass{article}
\usepackage{amsthm,amsmath,amssymb}
\RequirePackage[numbers]{natbib}
\RequirePackage[colorlinks,citecolor=blue,urlcolor=blue]{hyperref}

\def\bR{\mathbb R}
\def\bE{\mathbb E}
\def\R{\mathbb R}

\newtheorem{Theorem}{Theorem}
\newtheorem{theorem}{Theorem}
\newtheorem{lemma}[Theorem]{Lemma}
\newtheorem{Corollary}[Theorem]{Corollary}

\newtheorem{Assumption}[Theorem]{Assumption}
\newtheorem{remark}{Remark}
\def\P{{\mathbb P}}
\def\AA{{\cal A} }

\def\argmin{\mathop{\rm arg\,min}}

\def\rank{{\rm rank}}

\def\Sum{\overset{n}{\underset{i=1}{\sum}}}

\def\mR{\mathcal{R}}

\def\bmR{\mathbb R^{m_1\times m_2}}
\def \mI{\mathcal{I}}
\begin{document}

\title{Robust Matrix Completion}
\author{\textbf{ Olga Klopp} \\ CREST and MODAL'X\\University Paris Ouest, 92001 Nanterre, France\\\\
\textbf{ Karim Lounici}\\School of Mathematics, Georgia Institute of Technology\\
Atlanta, GA 30332-0160, USA\\\\and\\\\
\textbf{ Alexandre B. Tsybakov} \\CREST-ENSAE, UMR CNRS 9194,\\
3, Av. Pierre Larousse, 92240 Malakoff, France}
\date{}
\maketitle


\begin{abstract}
This paper considers the problem of estimation of a low-rank matrix when most of its entries are not observed and some of the observed entries are corrupted. The observations are noisy realizations of a sum of a low-rank matrix, which we wish to estimate, and a second matrix having  a complementary sparse structure such as elementwise sparsity or columnwise sparsity. We analyze a class of estimators obtained as solutions of a constrained convex optimization problem combining the nuclear norm penalty and a convex relaxation penalty for the sparse constraint.  Our assumptions allow for  simultaneous  presence of random and deterministic patterns in the sampling scheme. We establish rates of convergence for the low-rank component from partial and corrupted observations in the presence of noise and we show that these rates are minimax optimal up to logarithmic factors.
\end{abstract}

%



\section{Introduction}\label{introduction}


In the recent years, there have been a considerable interest in statistical inference for high-dimensional matrices. One particular problem is matrix completion where one observes only a small number $N \ll m_1m_2$ of the entries of a high-dimensional $m_1\times m_2$ matrix $L_0$ of rank $r$ and aims at inferring the missing entries.  In general, recovery of a matrix from a small number of observed entries
 is impossible, but, if the unknown matrix has low rank, then accurate and even exact recovery is possible. In the  noiseless setting, \cite{Candes_Tao,Gross, recht}
 established the following remarkable result: assuming that the matrix $L_0$ satisfies some low coherence condition, this matrix  can be recovered exactly by a constrained nuclear norm minimization with high probability from only $N \gtrsim  r \max\{m_1,m_2\} \log^2(m_1+m_2)$ entries observed uniformly at random.  A more common situation in applications corresponds to
 the noisy setting in which the few available entries are corrupted by noise. Noisy matrix completion has
 been in the focus of several recent studies (see, e.g., \cite{Keshavan,rhode-tsybakov-estimation,KLT11, wainwright-weighted,foygel-serebro-concentration,klopp_general, cai_max}).
 

The matrix completion problem is motivated by a variety of applications. An important question in applications is whether or not  matrix completion procedures are robust to corruptions.  Suppose that we  observe noisy entries of $A_0 = L_0 + S_0$ where $L_0$ is an unknown low-rank matrix and $S_0$ corresponds to some gross/malicious corruptions. We wish to recover $L_0$ but we observe only few entries of $A_0$ and, among those, a fraction happens to be corrupted by~$S_0$. Of course, we do not know which entries are corrupted. It has been shown empirically that uncontrolled and potentially adversarial gross errors affecting only a small portion of observations can be particularly harmful. For example, Xu et al. \cite{Xu_Caramanis_Sanghavi} showed that a very popular matrix completion procedure using nuclear norm minimization can fail dramatically even if $S_0$ contains only a single nonzero column.   It is particularly relevant in applications to recommendation systems where malicious users try to manipulate the outcome of matrix completion algorithms by introducing spurious perturbations $S_0$.  Hence, there is a need for new matrix completion techniques that are robust to the presence of corruptions $S_0$.

With this motivation, we consider the following setting of \textit{robust matrix completion}. Let $A_0\in \mathbb{R}^{m_{1}\times m_{2}}$ be an unknown matrix that can be represented as a sum $A_0=L_0+S_0$ where $L_0$ is a low-rank matrix and $S_0$ is a matrix with some low complexity structure such as entrywise sparsity or columnwise sparsity. We consider the observations $(X_i,Y_i), i=1,\dots,N,$ satisfying the trace regression model
\begin{equation}\label{model}
Y_{i}=\mathrm{tr}(X_{i}^{T}A_{0})+\xi_{i}, \:i=1,\dots,N,
\end{equation}
where $\mathrm{tr}({\rm M})$ denotes the trace of matrix ${\rm M}$. Here, the noise variables $\xi_{i}$ are independent and centered, and $X_{i}$ are $m_1\times m_2$ matrices taking values in the set
\begin{equation}\label{basisUSR}
\mathcal{X} = \left \{e_j(m_1)e_k^{T}(m_2),1\leq j\leq m_1, 1\leq k\leq m_2\right \},
\end{equation}
where $e_l(m)$, $l=1,\dots,m,$ are the canonical basis vectors in $\bR^{m}$. Thus, we observe some entries of matrix $A_0 $ with random noise.  Based on the observations $(X_i,Y_i)$, we wish to obtain accurate estimates of the components $L_0$  and $S_0$ in the high-dimensional setting $N\ll m_1m_2$. Throughout the paper, we assume that $(X_1,\dots,X_n)$ is independent of $(\xi_1,\dots,\xi_n)$.

We assume  that the set of indices $i$ of our $N$ observations is the union of two disjoint components $\Omega$ and $ \tilde \Omega$.  The first component $\Omega$ corresponds to  the ``non-corrupted''  noisy entries of  $L_0$, i.e., to the
 observations, for which the entry of $S_0$ is zero. The second set $\tilde \Omega$ corresponds to the observations, for which the 
  entry of $S_0$ is nonzero.  Given an observation, we do not know whether it belongs to the corrupted or non-corrupted part of the observations and we have 
  $\vert  \Omega\vert+\vert \tilde \Omega\vert=N$, where $\vert  \Omega\vert$ and $\vert \tilde \Omega\vert$ are non-random numbers of non-corrupted and corrupted observations, respectively.

A particular case of this setting is  the matrix decomposition problem where  $N=m_1m_2$, i.e., we observe all entries of $A_0$. Several recent works consider the matrix decomposition problem, mostly in the noiseless setting, $\xi_i \equiv 0$. Chandrasekaran et al.  \cite{Chandra_Parillo_sujay_Willsky}  analyzed the case when the matrix $S_0$ is sparse,  with small number of non-zero entries. They proved that exact recovery of $(L_0,S_0)$ is possible with high probability under additional identifiability conditions. This model was further studied by Hsu et al. \cite{HsuKakadeZhang} who give milder conditions for the exact recovery of $(L_0,S_0)$. Also in the noiseless setting, Candes et al. \cite{Candes_robust} studied the same model but with  positions of  corruptions chosen uniformly at random.  Xu et al. \cite{Xu_Caramanis_Sanghavi} studied a  model, in which the matrix $S_0$ is columnwise sparse with sufficiently small number of non-zero columns. Their method guarantees approximate recovery for the non-corrupted columns of the low-rank component~$L_0$.  Agarwal et al.  \cite{AgarwalNegahbanWainwright} consider a general model, in which the observations are noisy realizations of a linear transformation of 
$A_0$. Their setup includes the matrix decomposition problem and some other statistical models of interest but does 
not cover the matrix completion problem. Agarwal et al. \cite{AgarwalNegahbanWainwright} state a general result on approximate recovery of the pair $(L_0,S_0)$ imposing a ``spikiness condition'' on the low-rank
component $L_0$. Their analysis includes as particular cases both the  entrywise corruptions and the columnwise corruptions.

The robust matrix completion setting, when $N<m_1m_2$, was first considered by   Candes et al. \cite{Candes_robust}  in the noiseless case for entrywise sparse $S_0$. Candes et al. \cite{Candes_robust} assumed that the support of $S_0$ is selected uniformly at random and that $N$ is equal to $0.1m_1m_2$ or to some other fixed fraction of $m_1m_2$. Chen et al. \cite{ChenXuCaramanisSanghavi} considered also the noiseless case but with columnwise sparse $S_0$.  They proved that the same procedure as in \cite{Chandra_Parillo_sujay_Willsky} can recover the non-corrupted columns of $L_0$  and identify the set of indices of the corrupted columns.  This was done under the following assumptions: the locations of the non-corrupted columns are chosen uniformly at random; $L_0$ satisfies some sparse/low-rank incoherence condition; the total number of corrupted columns is small and a sufficient number of non-corrupted entries is observed. More recently, Chen et al. \cite{chen_jalali} and Li \cite{li} considered noiseless robust matrix completion with  entrywise sparse $S_0$. They proved exact recovery of the low-rank component  under an  incoherence condition on $L_0$ and some additional assumptions on the number of corrupted observations.

To the best of our knowledge, the present paper is the first study of robust matrix completion with noise. Our analysis is general and covers in particular the cases of 
columnwise sparse corruptions and entrywise sparse corruptions.  It is important to note that we do not require strong assumptions on the unknown matrices, such as the incoherence condition, or additional restrictions on the number of corrupted observations as in the noiseless case. This is due to the fact that we do not aim at exact recovery of the unknown matrix.  We emphasize that we do not need to know the rank of $L_0$ nor the sparsity level of $S_0$. We do not need to observe all entries of $A_0$ either. We only need to know an upper bound on the maximum of the absolute values of the entries of $L_0$ and $S_0$. Such information is often available in applications; for example, in recommendation systems, this bound is just the maximum rating. Another important point is that our method allows us to consider quite general and unknown sampling distribution. All the previous works on noiseless robust matrix completion assume the uniform sampling distribution. However, in practice the observed entries 
 are not guaranteed to follow the uniform scheme and the sampling distribution is not exactly known.

 We establish oracle inequalities  for the cases of entrywise sparse and columnwise sparse $S_0$.  For example, in the case of columnwise corruptions,
  we prove the following bound on the normalized Frobenius error of our estimator $(\hat L, \hat S)$ of $(L_0, S_0)$: with high probability
$$
\frac{\|\hat L - L_0\|_2^2}{m_1m_2}+\dfrac{\Vert  S_0-\hat S\Vert_2^{2}}{m_1m_2} \lesssim \dfrac{r\,\max(m_1,m_2)+\vert \tilde \Omega\vert}{\vert \Omega\vert}+\dfrac{s}{m_2}
$$
 where the symbol $\lesssim$ means that the inequality holds up to a multiplicative absolute constant and a factor, which is logarithmic  in $m_1$ and $m_2$. Here, $r$ denotes the rank of $L_0$, and $s$ is 
 the number of corrupted columns.
  Note that, when the number of corrupted columns $s$ and the proportion of corrupted observations $\vert \tilde \Omega\vert/ \vert  \Omega\vert$ are small, this bound implies that $O(r\,\max(m_1,m_2))$ observations are enough for successful and robust to corruptions matrix completion.  We also show that, both under the columnwise corruptions and entrywise corruptions, the obtained rates of convergence are minimax optimal up to logarithmic factors.

This paper is organized as follows. Section \ref{notations} contains the notation and definitions. We introduce our estimator in Section \ref{method} and we state the assumptions on the sampling scheme in Section \ref{sampling}. Section \ref{upper bounds} presents a general upper bound for the estimation error. In Sections \ref{column-wise_sparsity} and \ref{element-wise_sparsity}, we specialize this bound to the settings with columnwise corruptions and entrywise corruptions, respectively. In Section \ref{lower_bounds}, we prove that our estimator is minimax rate optimal up to a logarithmic factor. The Appendix contains the proofs.

\section{Preliminaries}\label{Preliminaries}

\subsection{Notation and definitions}\label{notations}

 
 {\it General notation.} For any set $I$, $\vert I\vert$ denotes its cardinality and $\bar I$ its complement. We write $a\vee b=\max(a,b)$ and $a\wedge b=\min(a,b)$. 

For a matrix $A$, $A^{i}$ is its $i$th column and $A_{ij}$ is its $(i,j)-$th entry. Let $I\subset \{1,\dots m_1\}\times\{1,\dots m_2\}$ be a subset of indices. Given a matrix $A$, we denote by $A_{I}$ its restriction on $I$, that is, $\left (A_{I}\right )_{ij}=A_{ij}$ if $(i,j)\in I$ and $\left (A_{I}\right )_{ij}=0$ if $(i,j)\not\in I$. 
In what follows, $\mathbf{Id}$ denotes the matrix of ones, i.e., $\mathbf{Id}_{ij}=1$ for any $(i,j)$ and ${\bf 0}$ denotes the
zero matrix, i.e.,  ${\bf 0}_{ij}=0$ for any $(i,j)$.

 For any $p\geq 1$, we denote by $\Vert \cdot\Vert_{p}$ the usual $l_p-$norm.  Additionally, we use the following matrix norms: $\Vert A\Vert_{*}$ is the nuclear norm (the sum of singular values), $\Vert A\Vert$ is the operator norm (the largest singular value), $\Vert A\Vert_{\infty}$ is the largest absolute value of the entries:  
 $$\|A\|_{\infty} = \max_{1\leq j \leq m_1,1\leq k \leq m_2}|A_{jk}|,$$
 the norm $\Vert A\Vert_{2,1}$ is the sum of $l_2$ norms of the columns of $A$ and $\Vert A\Vert_{2,\infty}$ is the largest $l_2$ norm of the columns of $A$:
$$
\|A\|_{2,1} =\sum_{k=1}^{m_2} \|A^{k}\|_2\quad \text{and}\quad \|A\|_{2,\infty} =\max_{1\leq k\leq m_2} \|A^{k}\|_2.
$$ 
 The inner product of matrices $A$ and $B$ is defined by $\langle  A,B  \rangle = \mathrm{tr}(AB^\top)$. 
 
%
\medskip
\textit{Notation related to corruptions.}
We first introduce the index sets $\mI$ and $\tilde\mI$. These are subsets of $\{1,\dots,m_1\}\times \{1,\dots,m_2\}$ that  are defined differently for the settings with columnwise sparse and entrywise sparse corruption matrix $S_0$. 

For the columnwise sparse matrix $S_0$, we define  
\begin{equation}\label{index_set_columns_1}
  \tilde \mI=\{1,\dots,m_1\}\times J \end{equation}
   where $J\subset \{1,\dots,m_2\}$ is the set of indices of the non-zero columns of $S_0$. For the entrywise sparse matrix $S_0$, we denote by $\tilde \mI$ 
   the set of indices of the non-zero elements of $S_0$. In both settings, $\mI$ denotes the complement of $\tilde \mI$.
   
  Let $\mR\,:\,\mathbb{R}^{m_1\times m_2}\rightarrow\mathbb{R}_{+}$ be a norm that will be used as a regularizer relative to the corruption matrix $S_0$.
The associated dual norm is defined by the relation
\begin{equation}
\mR^{*}(A)=\underset{\mR(B)\leq 1}{\sup}\langle A,B\rangle.
\end{equation}
Let $\vert A\vert$ denote the matrix whose entries are the absolute values of the entries of matrix $A$. The norm $\mR(\cdot)$ is called \textit{absolute} if it depends only on the absolute values of the entries of $A$:
$$\mR(A)=\mR(\vert A\vert).$$
For instance,  the $l_p$-norm and the $\|\cdot\|_{2,1}$-norm are absolute.  We call $\mR(\cdot)$ \textit{monotonic} if $\vert A\vert\leq \vert B\vert$ implies $\mR(A)\leq \mR(B)$. Here and below, the inequalities between matrices are understood as entry-wise inequalities. Any absolute norm is monotonic and vice versa (see, e.g., \cite{bauer}). 

\medskip
\textit{Specific notation.}
\begin{itemize}
\item We set $d=m_1+m_2$, $m=m_1\wedge m_2$, and $M=m_1\vee m_2$.
\item  Let $\{\epsilon_i\}_{i=1}^{n}$ be a sequence of i.i.d. Rademacher random variables. 
We define the following random variables called the stochastic terms:
\begin{equation*}\label{stoch1}
\Sigma_R=\dfrac{1}{n} \sum_{i\in \Omega} \epsilon_i X_i,\quad\Sigma=\dfrac{1}{N} \sum_{i\in \Omega} \xi_iX_i,\quad\text{and}\qquad W=\dfrac{1}{N} \sum_{i\in \Omega} X_i.
\end{equation*}
\item We denote by $r$ the rank of matrix $L_0$.
\item We denote by $N$ the number of observations, and by $n=\vert \Omega\vert$  the number of non-corrupted observations. The number of corrupted observations is $\vert \tilde\Omega\vert = N-n$. We set $\ae=N/n$. 
\item We use the generic symbol $C$ for positive constants that
do not depend on $n, m_1, m_2,r,s$ and can take different values at different appearances.
\end{itemize}

 
\subsection{Convex relaxation for robust matrix completion}\label{method}

For the usual matrix completion, i.e., when the corruption matrix $S_0=\bf{0}$, one of the most popular methods of solving the problem is based on constrained nuclear norm minimization. For example,  the following constrained matrix Lasso estimator is introduced in \cite{klopp_general}: 
\begin{equation*} 
\hat{A}\in\underset{
\left\Vert A\right\Vert_{\infty}\leq \mathbf{a}}{\argmin}\left \{\dfrac{1}{n}\Sum \left (Y_i-\left\langle X_i,A\right\rangle\right )^{2}+\lambda \Vert A\Vert_*\right \},
\end{equation*}
where $\lambda>0$ is a regularization parameter and $\mathbf{a}$ is an upper bound on $\left\Vert L_0\right\Vert_{\infty}$.

To account for the presence of non-zero corruptions $S_0$, we introduce an additional norm-based penalty that should be chosen depending on the structure of $S_0$.  
We consider the following estimator $(\hat{L},\hat S)$ of the pair $(L_0,S_0)$:
\begin{equation}\label{estimator}
(\hat{L},\hat S)\in\underset{^{\left\Vert L\right\Vert_{\infty}\leq \mathbf{a}}_{\left\Vert S\right\Vert_{\infty}\leq \mathbf{a}}
}{\argmin}\left \{\dfrac{1}{N}\sum^{N}_{i=1} \left (Y_i-\left\langle X_i,L+S\right\rangle\right )^{2}+\lambda_{1} \Vert L\Vert_* +\lambda_{2}\mathcal{R}(S)\right \}.
\end{equation}
Here $\lambda_1>0$ and $\lambda_2>0$ are regularization parameters and $\mathbf{a}$ is an upper bound on $\left\Vert L_0\right\Vert_{\infty}$ and
 $\left\Vert S_0\right\Vert_{\infty}$.  Note that this definition and all the proofs
  can be easily adapted to the setting with two different upper bounds for   $\left\Vert L_0\right\Vert_{\infty}$ and
   $\left\Vert S_0\right\Vert_{\infty}$ as it can be the case in some applications. Thus, the results of the paper extend to this case as well.

For the following two key examples of sparsity structure of $S_0$, we consider specific regularizers  $\mR$.
 \begin{itemize}
 \item \textbf{Example 1.} Suppose that $S_0$ is {\it columnwise sparse}, that is, it has a small number $s<m_2$ of non-zero columns.
  We use the $\Vert \cdot\Vert_{2,1}$-norm regularizer for such a sparsity structure: $\mR(S)=\Vert S\Vert_{2,1}.$ The associated dual norm is $\mR^{*}(S)=\Vert S\Vert_{2,\infty}$. 

 \item \textbf{Example 2.} Suppose now that $S_0$ is {\it entrywise sparse}, that is, that it has $s\ll m_1m_2$ non-zero entries. The usual choice of regularizer for such a sparsity structure is the $l_1$ norm: $\mR(S)=\Vert S\Vert_{1}$. The associated dual norm is $\mR^{*}(S)=\Vert S \Vert_{\infty}$. 
 \end{itemize}

 In these two examples, the regularizer $\mR$  is \textit{decomposable} with respect to a properly chosen 
  set of indices $I$. That is, for any matrix $A\in \bmR$ we have
 \begin{equation}\label{decomposability}
 \mR(A)=\mR(A_{I})+\mR(A_{\bar I}).
 \end{equation}
   For instance, the $\Vert \cdot\Vert_{2,1}$-norm is decomposable with respect to any set $I$ such that
   \begin{equation}\label{index_set_columns}
    I=\{1,\dots,m_1\}\times J\end{equation}
    where $J\subset \{1,\dots,m_2\}$.
The usual $l_1$ norm is decomposable with respect to any subset of indices
$I$.


\subsection{Assumptions on the sampling scheme and on the noise}\label{sampling}

In the literature on the usual matrix completion ($S_0=\bf{0}$), it is commonly assumed that the observations  $X_i$ are i.i.d.  For robust matrix completion, it is more realistic to assume the presence of 
two subsets in the observed $X_i$. The first subset $\{X_{i},\,i\in \Omega\}$ is a collection of i.i.d. random matrices with some unknown distribution on 
 \begin{equation}\label{basisUSR_random}
 \mathcal{X}' = \left \{e_j(m_1)e_k^{T}(m_2),(j,k)\in \mathcal{I}\right \}. 
 \end{equation}
 These $X_i$'s are of the same type as in the usual matrix completion. They are the $X$-components of non-corrupted observations (recall that the entries of $S_0$ corresponding to indices in $\mathcal{I}$ are equal to zero).
On this non-corrupted part of observations, we require some  assumptions on the sampling distribution
 (see Assumptions \ref{assPi}, \ref{L}, \ref{marginal_columns}, and \ref{marginal_sparse} below). 
 
 The second subset $\{X_{i},\;i\in \tilde \Omega\}$  is a collection of matrices with values in
 \begin{equation*} 
 \mathcal{X}'' = \left \{e_j(m_1)e_k^{T}(m_2),(j,k)\in  \tilde\mI\right \}.
 \end{equation*}
 These are the $X$-components of corrupted observations. Importantly, we \textit{make no assumptions} on how they  are sampled.  Thus, for any $i\in \tilde\Omega$, we have that the index of the corresponding entry belongs to $\tilde \mI$ and we make no further assumption.
 If we take the example of recommendation systems, this partition into  $\{X_{i},\,i\in \Omega\}$ and $\{X_{i},\,i\in \tilde \Omega\}$ accounts for the difference in behavior of  normal and malicious users.

   As there is no hope for  recovering the unobserved entries of $S_0$, one should consider only the estimation of the restriction of $S_0$ to $\tilde \Omega$. This is equivalent to assume that we estimate the whole $S_0$ when all unobserved entries of $S_0$  are equal to zero, cf. \cite{ChenXuCaramanisSanghavi}. This assumption will be done throughout the paper. 
  
    For $i\in \Omega$, we suppose that
 $X_i$ are i.i.d realizations of a random matrix $X$ having distribution $\Pi$ on the set $ \mathcal{X}'$.
   Let $\pi_{jk}=\mathbb{P}\left (X=e_j(m_1)e_k^{T}(m_2)\right )$ be the probability to observe the $(j,k)$-th entry. One of the particular settings of this problem is the case of the uniform on $\mathcal{X}'$ distribution $\Pi$. It was previously considered in the context of noiseless robust matrix completion, see, e.g., \cite{ChenXuCaramanisSanghavi}.  We consider here a more general   sampling model. In particular, 
   we suppose that any non-corrupted element is sampled with positive probability:
\begin{Assumption}\label{assPi}
There exists a positive constant $\mu\geq 1$ such that, for any $(j,k)\in \mathcal{I}$, $$\pi_{jk}\geq (\mu \vert \mI\vert)^{-1}.$$
\end{Assumption}
If $\Pi$ is the of uniform distribution on $\mathcal{X}'$ we have $\mu=1$.
For $A\in\mathbb{R}^{m_1\times m_2}$  set 
$$ \Vert A\Vert _{L_2(\Pi)}^{2}=\bE\left(\langle A,X\rangle^{2}\right).
$$
 Assumption \ref{assPi} implies that
\begin{equation}\label{ass1}
 \Vert A\Vert^{2} _{L_2(\Pi)}\geq (\mu\,\vert \mI\vert)^{-1}\Vert A_{\mathcal{I}}\Vert^{2} _{2}.
 \end{equation}
Denote by $\pi_{\cdot k}=\underset{j=1}{\overset{m_1}{\Sigma}}\pi_{jk}$ the probability to observe an element from the $k$-th column and by $\pi_{j\cdot }=\underset{k=1}{\overset{m_2}{\Sigma}}\pi_{jk}$ the probability to observe an element from the $j$-th row. 
The following assumption requires that no column and no row is sampled with too high probability.
 \begin{Assumption}\label{L}
 There exists a positive constant  $L\geq 1$ such that
 \begin{equation*} 
 \underset{i,j}{\max}\left (\pi_{\cdot k},\pi_{j\cdot }\right )\leq L/m.
 \end{equation*}
 \end{Assumption}  
 
 This assumption will be used in Theorem \ref{thm2} below. In Sections \ref{column-wise_sparsity} and \ref{element-wise_sparsity}, we apply Theorem \ref{thm2} to the particular cases of columnwise sparse and entrywise sparse corruptions. There, we will need more restrictive assumptions on the sampling distribution (see Assumptions \ref{marginal_columns} and \ref{marginal_sparse}).
 
 We assume below that the noise variables $\xi_i$ are sub-gaussian:
     
   \begin{Assumption}\label{noise} There exist positive constants
   $\sigma$ and $c_1$ such that
      $$\underset{i=1,\dots,n}{\max}\bE\exp\left (\xi^{2}_i/\sigma^{2}\right )< c_1. $$
  \end{Assumption}


\section{Upper bounds for general regularizers}\label{upper bounds}

 In this section we state our main result which applies to a general convex program \eqref{estimator} where $\mR$ is an absolute norm and a decomposable regularizer.  In the next sections, we consider in detail two particular choices, $\mR(\cdot)=\|\cdot\|_{1}$ and $\mR(\cdot)=\|\cdot\|_{2,1}$. Introduce the notation:
\begin{equation}\label{psi}
\begin{split}
\Psi_{1}&= \mu^{2}\,m_1m_2\,r\,\left (\ae^{2}\lambda^{2}_1+\mathbf{a}^{2}\left ( \bE\left ( \Vert\Sigma_R\Vert\right )\right )^{2}\right )+\mathbf{a}^{2}\,\mu\,\sqrt{\dfrac{\log(d)}{n}},\\
\Psi_2&= \mu\,\mathbf{a}\,\mR(\mathbf{Id}_{\tilde{\Omega}})\left (\dfrac{\lambda_2\,\mathbf{a}}{\lambda_1}\bE\left ( \Vert\Sigma_R\Vert\right )+\ae\lambda_2+\mathbf a\,\bE\left ( \mR^{*}(\Sigma_R)\right )\right ),\\
\Psi_3&=\dfrac{\mu\,\vert \tilde\Omega\vert
\left (\mathbf a^{2}+\sigma^{2}\log(d)\right  )}{N}\left (\dfrac{\mathbf{a}\,\bE\left ( \Vert\Sigma_R\Vert\right )}{\lambda_1}+\dfrac{\mathbf{a}\,\bE \left (\mR^{*}(\Sigma_R)\right )}{\lambda_2}+\ae\right )+\dfrac{\mathbf{a}^{2}\vert \tilde{\mI}\vert}{m_1m_2},\\
\Psi_4&=\mu\,\mathbf{a}^{2}\,\sqrt{\dfrac{\log(d)}{n}}+\mu\,\mathbf{a}\,\mR(\mathbf{Id}_{\tilde{\Omega}})\left [\ae\,\lambda_{2}+ \mathbf a\,\bE\left ( \mR^{*}(\Sigma_R)\right )\right ]\\&\hskip 4 cm+\left [ \dfrac{\mathbf a\,\bE\left (\mR^{*}(\Sigma_R)\right )}{\lambda_2}+\ae\right ]\frac{\mu\,\vert \tilde\Omega\vert\left (\mathbf a^{2}+\sigma^{2}\log(d)\right  )}{N}
\end{split}
\end{equation}
where $d=m_1+m_2$.

\begin{theorem}\label{thm2}
Let $\mR$ be an absolute norm and a decomposable regularizer. Assume that $ \left\Vert L_0\right\Vert_{\infty}\leq \mathbf{a}$, $ \left\Vert S_0\right\Vert_{\infty}\leq \mathbf{a}$ for some constant $\mathbf{a}$ and let Assumptions \ref{assPi} - \ref{noise} be satisfied.
Let $\lambda_1>4\left\Vert \Sigma\right\Vert$, and $\lambda_2\geq 4\left (\mR^{*}(\Sigma)+2\mathbf{a}\mR^{*}(W)\right )$.  Then, with probability at least $1-4.5\,d^{-1}$,
\begin{equation}\label{upper_bound}
 \begin{split}
 \dfrac{\Vert  L_0-\hat L\Vert_2^{2}}{m_1m_2}+\dfrac{\Vert S_0-\hat S\Vert_2^{2}}{m_1m_2}&\leq C\,\left \{ \Psi_1+ \Psi_2+\Psi_3\right \}
 \end{split}
 \end{equation}
 where $C$ is an absolute constant. Moreover, with the same probability,
\begin{equation}\label{upper_bound1}
\begin{split}
\dfrac{\Vert \hat S_{\mI}\Vert^{2}_2}{\vert\mI\vert}\leq C\Psi_4.
\end{split}
\end{equation}
\end{theorem}
 The term $\Psi_{1}$ in \eqref{upper_bound} corresponds to the estimation error associated with matrix completion of a rank $r$ matrix. The second and the third terms account for the error induced by corruptions.
In the next two sections we apply Theorem~\ref{thm2} to the settings with the entrywise sparse and columnwise sparse corruption matrices  $S_0$.

\section{Columnwise sparse corruptions}\label{column-wise_sparsity}

In this section, we assume that that $S_0$ has at most $s$ non-zero columns,
 and $s\leq m_2/2$.  We use here the $\Vert \cdot\Vert_{2,1}$-norm regularizer $\mR$.  Then, the convex program~\eqref{estimator} takes form
\begin{equation}\label{estimator_column}
(\hat{L},\hat S)\in\underset{^{\left\Vert L\right\Vert_{\infty}\leq \mathbf{a}}_{\left\Vert S\right\Vert_{\infty}\leq \mathbf{a}}
}{\argmin}\left \{\dfrac{1}{N}\sum^{N}_{i=1} \left (Y_i-\left\langle X_i,L+S\right\rangle\right )^{2}+\lambda_{1} \Vert L\Vert_1 +\lambda_{2}\Vert S\Vert_{2,1}\right \}.
\end{equation}
 Since $S_0$ has at most $s$ non-zero columns, we have 
 $\vert \tilde{\mI}\vert=m_1s$. Furthermore, by the Cauchy-Schwarz  inequality,  $\Vert\mathbf{Id}_{\tilde{\Omega}}\Vert_{2,1}\leq \sqrt{s\vert \tilde \Omega\vert}$. Using these remarks we replace $\Psi_2$, $\Psi_3$ and $\Psi_4$  by the larger quantities
\begin{equation*}
\begin{split}
\Psi'_2&= \mu\,\mathbf{a}\,\sqrt{s\vert \tilde \Omega\vert}\left (\dfrac{\mathbf a\,\lambda_2}{\lambda_1}\bE\left ( \Vert\Sigma_R\Vert\right )+\ae\lambda_2+\mathbf a\,\bE \Vert\Sigma_R\Vert_{2,\infty}\right ),\\
\Psi'_3&=\dfrac{\mu\,\vert \tilde\Omega\vert
\left (\mathbf a^{2}+\sigma^{2}\log(d)\right  )}{N}\left (\dfrac{\mathbf a\bE\left ( \Vert\Sigma_R\Vert\right )}{\lambda_1}+\dfrac{\mathbf a\bE \Vert\Sigma_R\Vert_{2,\infty}}{\lambda_2}+\ae\right )+\dfrac{\mathbf{a}^{2}s}{m_2},
\\
\Psi'_4&=\mu\,\mathbf{a}^{2}\,\sqrt{\dfrac{\log(d)}{n}}+\mu\,\mathbf{a}\,\sqrt{s\vert \tilde \Omega\vert}\left [\ae\,\lambda_{2}+ \mathbf a\,\bE\Vert\Sigma_R\Vert_{2,\infty}\right ]\\&\hskip 5 cm+\left [ \dfrac{\mathbf a\,\bE\Vert\Sigma_R\Vert_{2,\infty}}{\lambda_2}+\ae\right ]\frac{\mu\,\vert \tilde\Omega\vert\left (\mathbf a^{2}+\sigma^{2}\log(d)\right  )}{N}.
\end{split}
\end{equation*}
 Specializing Theorem \ref{thm2} to this case yields the following corollary.
 
 
\begin{Corollary}\label{corollary_column_1}
Assume that $\left\Vert L_{0}\right\Vert_{\infty}\leq \mathbf{a}$ and $\left\Vert S_{0}\right\Vert_{\infty}\leq \mathbf{a}$. 
Let the regularization parameters $(\lambda_1,\lambda_2)$ satisfy 
$$\lambda_1>4\left\Vert \Sigma\right\Vert\;\text{and} \;\lambda_2\geq 4\left (\Vert\Sigma\Vert_{2,\infty}+2\mathbf{a}\Vert W\Vert_{2,\infty}\right ).
$$ 
Then, with probability at least $1-4.5\,d^{-1}$, for any solution $(\hat L,\hat S)$ of the convex program \eqref{estimator_column} with such regularization parameters $(\lambda_1,\lambda_2)$ we have
\begin{equation*}
 \begin{split}
 \dfrac{\Vert L_0-\hat L\Vert_2^{2}}{m_1m_2}+\dfrac{\Vert  S_0-\hat S\Vert_2^{2}}{m_1m_2}&\leq C\,\left \{ \Psi_1+ \Psi'_2+\Psi'_3\right \}.
 \end{split}
 \end{equation*}
 where $C$ is an absolute constant. Moreover, with the same probability,
\begin{equation*}
\begin{split}
\dfrac{\Vert \hat S_{\mI}\Vert^{2}_2}{\vert\mI\vert}\leq C\Psi'_4.
\end{split}
\end{equation*}
\end{Corollary}

In order to get a bound in a closed form, we need to obtain suitable upper bounds on the stochastic terms $\Sigma$, $\Sigma_{R}$ and $W$. We derive such bounds under an additional assumption on the column marginal sampling distribution. 
 Set $\pi_{\cdot,k}^{(2)} =\sum_{j=1}^{m_1} \pi_{jk}^2$.   
\begin{Assumption}\label{marginal_columns}
 There exists a positive constant $\gamma\geq 1$ such that 
$$
\underset{k}{\max}\;\pi_{\cdot,k}^{(2) } \leq \frac{\gamma^{2}}{\vert \mI\vert\,m_2}.
$$
\end{Assumption}
This condition prevents the columns from being sampled with too high probability and guarantees that the non-corrupted observations are well spread out among the columns.  Assumption \ref{marginal_columns} is clearly less restrictive than assuming that $\Pi$ is  uniform as it was done in the previous work on noiseless robust matrix completion. In particular, Assumption \ref{marginal_columns} is satisfied when the distribution $\Pi$ is approximately uniform, i.e., when $\pi_{jk} \asymp \frac{1}{m_1(m_2-s)}$. 
Note that Assumption \ref{marginal_columns} implies the following milder condition on the marginal sampling distribution: \begin{equation}\label{milder-marginal}
\underset{k}{\max}\;\pi_{\cdot k}\leq \frac{\sqrt{2}\,\gamma}{m_2}.
\end{equation}
 Condition \eqref{milder-marginal} is sufficient to control $\|\Sigma\|_{2,\infty}$ and $\|\Sigma_R\|_{2,\infty}$ while to we need a stronger Assumption \ref{marginal_columns} to control $\|W\|_{2,\infty}$.   
 
 The following lemma gives the order of magnitude of the stochastic terms driving the rates of convergence.
 
 
\begin{lemma}\label{lemma_stocastique} Let the distribution $\Pi$ on $\mathcal{X}'$ satisfy Assumptions  \ref{assPi}, \ref{L} and \ref{marginal_columns}. Let also Assumption \ref{noise} hold. Assume that $ N\leq m_1m_2$, $n\leq \vert \mI\vert$, and  $\log m_2 \geq 1$. Then, there exists an absolute constant $C>0$ such that, for any $t>0$, the following bounds on the norms of the stochastic terms hold with probability at least $1-e^{-t}$, as well as the associated bounds in expectation.
\begin{equation*}
\begin{split}
(i) \quad&\left\Vert \Sigma\right\Vert\leq C\sigma\max\left(\sqrt{\dfrac{L(t+\log d)}{\ae\,Nm}}, \frac{(\log m)(t+\log d)}{N}\right)\quad\text{and}\\
&\bE \left\Vert \Sigma_{R}\right\Vert\leq C\left (\sqrt{\dfrac{L\log(d)}{nm}}+\frac{\log^{2} d}{N}\right );\\
  (ii) \quad &\left\Vert \Sigma\right\Vert_{2,\infty}\leq C\sigma \left( \sqrt{\frac{\gamma(t+\log(d))}{\ae N m_2} }+ \frac{t+\log d}{N}\right)\quad\text{and}\\
    &\bE\left\Vert \Sigma\right\Vert_{2,\infty}\leq C\sigma\,\left( \sqrt{\frac{\gamma\log(d)}{\ae N m_2} }+ \frac{\log d}{N}\right);\\
    (iii) \quad &\left\Vert \Sigma_R\right\Vert_{2,\infty}\leq C\left( \sqrt{\frac{\gamma(t+\log(d))}{n m_2} }+ \frac{t+\log d}{n}\right)\quad\text{and}\\
         &\bE\left\Vert \Sigma_R\right\Vert_{2,\infty}\leq C \left( \sqrt{\frac{\gamma\log(d)}{n m_2} }+ \frac{\log d}{n}\right);\\
    \end{split}
    \end{equation*}
    \begin{equation*}
     \begin{split}
        (iv) \quad &\Vert W\Vert_{2,\infty}\leq C  \left(\frac{\gamma (t+\log m_2)^{1/4}}{\sqrt{\ae N m_2}}\left (1+\sqrt{\frac{m_2(t+\log m_2)}{n}}\right )^{1/2} + \frac{t+\log m_2}{N}   \right)\\
          & \bE \Vert W\Vert_{2,\infty}\leq C \left(\frac{\gamma \log^{1/4}( d)}{\sqrt{\ae N m_2}}\left(1+ \sqrt{\frac{m_2\log d}{n}}\right )^{1/2} + \frac{\log d}{N}   \right).
        \end{split}
    \end{equation*}
\end{lemma}

Let
\begin{equation}\label{def_n}
n^{*}=2\,\log(d) \left (\dfrac{m_2}{\gamma}\vee\frac{m\,\log^{2}m}{L}\right ).
\end{equation}
Recall that $\ae=\frac{N}{n}\geq 1$. If $n\geq n^{*}$, using the bounds given by Lemma \ref{lemma_stocastique}, we can chose the regularization parameters $\lambda_1$ and $\lambda_2$ in the following way:
\begin{equation}\label{reg_parameters_column}
\lambda_1=C\left (\sigma\vee\mathbf{a}\right )\sqrt{\dfrac{L\,\log(d)}{Nm}}\quad\text{and} \quad
\lambda_2=C\,\gamma\left (\sigma\vee\mathbf{a}\right )\sqrt{\dfrac{\log(d)}{Nm_2}},
\end{equation}
where $C>0$ is a large enough numerical constant.

With this choice of the regularization parameters, Corollary \ref{corollary_column_1} implies the following result.
\begin{Corollary}\label{upper_bound_column}
 Let the distribution $\Pi$ on $\mathcal{X}'$ satisfy Assumptions  \ref{assPi}, \ref{L} and \ref{marginal_columns}. Let  Assumption \ref{noise} hold and $\left\Vert L_{0}\right\Vert_{\infty}\leq \mathbf{a}$, $\left\Vert S_{0}\right\Vert_{\infty}\leq \mathbf{a}$. Assume that $ N\leq m_1m_2$ and $n^{*}\leq n$. Then, with probability at least $1-6/d$ for any solution $(\hat L,\hat S)$ of the convex program \eqref{estimator_column} with the regularization parameters $(\lambda_1,\lambda_2)$ given by \eqref{reg_parameters_column}, we have
 \begin{equation} \label{upper_bound_columns_1}
 \begin{split}
 \dfrac{\Vert L_0-\hat L\Vert_2^{2}}{m_1m_2}+\dfrac{\Vert  S_0-\hat S\Vert_2^{2}}{m_1m_2}&\leq C_{\mu,\gamma, L}(\sigma\vee\mathbf{a})^{2}\log (d)\,\ae \dfrac{r\,M+\vert\tilde \Omega\vert}{n}+\dfrac{\mathbf{a}^{2}s}{m_2}
 \end{split}
 \end{equation}
where $C_{\mu,\gamma,L}>0$ can depend only on $\mu,\gamma,L$. Moreover, with the same probability,
\begin{equation*}
\begin{split}
\dfrac{\Vert \hat S_{\mI}\Vert^{2}_2}{\vert\mI\vert}\leq C_{\mu,\gamma, L} \dfrac{\ae(\sigma\vee\mathbf{a})^{2}\,\vert\tilde \Omega\vert\,\log (d)}{n}+\dfrac{\mathbf{a}^{2}s}{m_2}.
\end{split}
\end{equation*}
\end{Corollary}
 \textbf{Remarks.}
  \textbf{1.} The upper bound \eqref{upper_bound_columns_1} can be decomposed into two terms. The first term is proportional to $rM/n$.
 It  is of the same order as in the case  of the usual matrix completion, see \cite{KLT11,klopp_general}. The second term accounts for the corruption. It is proportional to the number of corrupted columns $s$ and to the number of corrupted observations $\vert\tilde \Omega\vert$. This term vanishes if there is no corruption, i.e., when $S_0=\bf{0}$.

  \textbf{2.} If all entries of $A_0$ are observed, i.e., the matrix decomposition problem is considered, the bound \eqref{upper_bound_columns_1} is analogous to the corresponding bound in  \cite{AgarwalNegahbanWainwright}. Indeed, then $\vert \tilde \Omega\vert=sm_1, N=m_1m_2$, $\ae \leq 2$ 
  and we get
  \begin{equation*} 
   \begin{split}
   \dfrac{\Vert L_0-\hat L\Vert_2^{2}}{m_1m_2}+\dfrac{\Vert  S_0-\hat S\Vert_2^{2}}{m_1m_2}&\lesssim\ (\sigma\vee\mathbf{a})^{2}\left (\dfrac{r\,M}{m_1m_2}+\dfrac{s}{m_2}\right ).
   \end{split}
   \end{equation*}
  The estimator studied in  \cite{AgarwalNegahbanWainwright} for matrix decomposition problem is similar to our program  \eqref{estimator_column}. The difference between these estimators is that in \eqref{estimator_column} the minimization is over $\Vert \cdot\Vert_{\infty}$-balls  while the program of \cite{AgarwalNegahbanWainwright} uses the minimization over $\Vert\cdot \Vert_{2,\infty}$-balls and requires the knowledge of a bound on the norm $\Vert L_0 \Vert_{2,\infty}$ of the unknown matrix $L_0$.

 \textbf{3.}
Suppose that the number of corrupted columns is small ($s\ll m_2$). Then, Corollary \ref{upper_bound_column} guarantees, that the prediction error  of our estimator is small whenever the number of non-corrupted observations $n$ satisfies the following condition 
\begin{equation}\label{condition_small_error}
n\gtrsim(m_1\vee m_2)\rank(L_0)+\vert\tilde \Omega\vert
\end{equation}
where $\vert\tilde \Omega\vert$ is the number of corrupted observations. This quantifies the sample size  sufficient for successful (robust to corruptions) matrix completion. 
  When the rank $r$ of $L_0$ is small and $s\ll m_2$, the right hand side of \eqref{condition_small_error}  is considerably smaller than the total number of entries $m_1m_2$.

  \textbf{4.} By changing the numerical constants, 
one can obtain that the upper bound \eqref{upper_bound_columns_1} is valid with probability $1-6d^{-\alpha}$ for a given  $\alpha\geq 1$.
  
\section{Entrywise sparse corruptions}\label{element-wise_sparsity}
We assume now that $S_0$ has $s$ non-zero entries but they do not necessarily lay in a small subset of columns.  We will also assume that $s\leq \frac{m_1m_2}{2}$.
We use now the $l_1$-regularizer $\mR$. Then the convex program \eqref{estimator} takes the form
\begin{equation}\label{estimator_sparse}
(\hat{L},\hat S)\in\underset{^{\left\Vert L\right\Vert_{\infty}\leq \mathbf{a}}_{\left\Vert S\right\Vert_{\infty}\leq \mathbf{a}}
}{\argmin}\left \{\dfrac{1}{N}\sum^{N}_{i=1} \left (Y_i-\left\langle X_i,L+S\right\rangle\right )^{2}+\lambda_{1} \Vert L\Vert_{\ast} +\lambda_{2}\Vert S\Vert_{1}\right \}.
\end{equation}
The support $\tilde\mI = \{ (j,k)\,:\, (S_0)_{jk} \neq 0\}$ of the non-zero entries of $S_0$ satisfies 
 $\vert \tilde{\mI}\vert=s$. Also, $\Vert\mathbf{Id}_{\tilde{\Omega}}\Vert_{1}=\vert \tilde \Omega\vert$ so that  $\Psi_2$, $\Psi_3$, and $\Psi_4$  take form
\begin{equation*}
\begin{split}
\Psi''_2&= \mu\,\mathbf{a}\,\vert \tilde \Omega\vert\left (\dfrac{\mathbf a\,\lambda_2}{\lambda_1}\bE\left ( \Vert\Sigma_R\Vert\right )+\ae\lambda_2+\mathbf a\,\bE \Vert\Sigma_R\Vert_{2,\infty}\right ),\\
\Psi''_3&=\dfrac{\mu\,\vert \tilde\Omega\vert
\left (\mathbf a^{2}+\sigma^{2}\log(d)\right  )}{N}\left (\dfrac{\mathbf a\bE\left ( \Vert\Sigma_R\Vert\right )}{\lambda_1}+\dfrac{\mathbf a\bE \Vert\Sigma_R\Vert_{2,\infty}}{\lambda_2}+\ae\right )+\dfrac{\mathbf{a}^{2}s}{m_1m_2},
\\
\Psi''_4&=\mu\,\mathbf{a}^{2}\,\sqrt{\dfrac{\log(d)}{n}}+\mu\,\mathbf{a}\,\vert \tilde \Omega\vert\left [\ae\,\lambda_{2}+ \mathbf a\,\bE\Vert\Sigma_R\Vert_{2,\infty}\right ]\\&\hskip 5 cm+\left [ \dfrac{\mathbf a\,\bE\Vert\Sigma_R\Vert_{2,\infty}}{\lambda_2}+\ae\right ]\frac{\mu\,\vert \tilde\Omega\vert\left (\mathbf a^{2}+\sigma^{2}\log(d)\right  )}{N}.
\end{split}
\end{equation*}
 Specializing Theorem \ref{thm2} to this case yields the following corollary:
\begin{Corollary}\label{corollary_sparse_1}
Assume that $\left\Vert L_{0}\right\Vert_{\infty}\leq \mathbf{a}$ and $\left\Vert S_{0}\right\Vert_{\infty}\leq \mathbf{a}$. 
Let the regularization parameters $(\lambda_1,\lambda_2)$ satisfy 
$$\lambda_1>4\left\Vert \Sigma\right\Vert\;\text{and} \;\lambda_2\geq 4\left (\Vert\Sigma\Vert_{\infty}+2\mathbf{a}\Vert W\Vert_{\infty}\right ).
$$ 
Then, with probability at least $1-4.5\,d^{-1}$, for any solution $(\hat L,\hat S)$ of the convex program \eqref{estimator_sparse} with such regularization parameters $(\lambda_1,\lambda_2)$ we have
%
%
\begin{equation*}
 \begin{split}
 \dfrac{\Vert L_0-\hat L\Vert_2^{2}}{m_1m_2}+\dfrac{\Vert S_0-\hat S\Vert_2^{2}}{m_1m_2}&\leq C\,\left \{ \Psi_1+ \Psi''_2+\Psi''_3\right \}
 \end{split}
 \end{equation*}
 where $C$ is an absolute constant. Moreover, with the same probability,
\begin{equation*}
\begin{split}
\dfrac{\Vert \hat S_{\mI}\Vert^{2}_2}{\vert\mI\vert}\leq C\Psi''_4.
\end{split}
\end{equation*}
\end{Corollary}
In order to get a bound in a closed form we need to obtain suitable upper bounds on the stochastic terms $\Sigma,\Sigma_{R}$ and $W$. We provide such bounds under the following additional assumption on the sampling distribution. 

\begin{Assumption}\label{marginal_sparse}
There exists a positive constant $\gamma\geq 1$ such that 
$$\underset{i,j}{\max}\;\pi_{ij}\leq \frac{\mu_1}{\vert \mI\vert}.$$  
\end{Assumption}

This assumption prevents any entry from being sampled too often and  guarantees that the observations are well spread out over the non-corrupted entries.
Assumptions \ref{assPi} and  \ref{marginal_sparse} imply that the sampling distribution  $\Pi$ is approximately uniform in the sense that $\pi_{jk} \asymp \frac{1}{\vert \mI\vert}$. In particular, since $\vert \mI\vert\leq \frac{m_1m_2}{2}$, Assumption \ref{marginal_sparse} implies Assumption \ref{L} for $L=2\mu_1$.

\begin{lemma}\label{lemma_stochastique-sparse} 
Let the distribution $\Pi$ on $\mathcal{X}'$ satisfy Assumptions  \ref{assPi}, and \ref{marginal_sparse}. Let also Assumption \ref{noise} hold. Then, there exists an absolute constant $C>0$ such that, for any $t>0$, the following bounds on the norms of the stochastic terms hold with probability at least $1-e^{-t}$, as well as the associated bounds in expectation.
\begin{equation*}\begin{split}
(i) \quad&\Vert W\Vert_{\infty}\leq  C \left( \frac{\mu_1}{\ae m_1m_2} + \sqrt{ \frac{\mu_1 (t+\log d))}{ \ae N m_1m_2}} + \frac{ t+\log d}{ N}\right)\quad\text{and}\\
 &\bE \Vert W\Vert_{\infty}\leq  C\left( \frac{\mu_1}{\ae m_1m_2} + \sqrt{ \frac{\mu_1 \log d}{ \ae N m_1m_2}} + \frac{\log d}{ N}\right);
\end{split}
  \end{equation*}
\begin{equation*}
 \begin{split}
 (ii) \quad&\left\Vert \Sigma\right\Vert_{\infty}\leq C  \sigma\left(\sqrt{\frac{\mu_1(t+\log d)}{\ae Nm_1m_2}} + \frac{t+\log d}{N} \right)\quad\text{and}\\
  &\bE\left\Vert \Sigma\right\Vert_{\infty}\leq C \sigma\left(  \sqrt{\frac{\mu_1 \log d}{\ae N m_1 m_2}} + \frac{ \log d}{N} \right);
 \end{split}
\end{equation*}
 \begin{equation*}
 \begin{split}
 (iii) \quad&\left\Vert \Sigma_R\right\Vert_{\infty}\leq C \left(\sqrt{\frac{\mu_1(t+\log d)}{n\,m_1m_2}} + \frac{t+\log d}{n} \right)\quad\text{and}\\
 &\bE\left\Vert \Sigma_R\right\Vert_{\infty}\leq C \left(\sqrt{\frac{\mu_1 \log d}{n\,m_1m_2}} + \frac{\log d}{n} \right).
 \end{split}
 \end{equation*}     
\end{lemma}

Using Lemma \ref{lemma_stocastique}(i), and Lemma \ref{lemma_stochastique-sparse}, under the conditions 
\begin{equation}\label{cond_n_spars}
 \begin{split}
\frac{m_1m_2\log d}{\mu_1}\geq n\geq  \frac{2\,m\log(d)\log^{2}(m)}{L} 
 \end{split}
 \end{equation}  
 we can choose the regularization parameters $\lambda_1$ and $\lambda_2$ in the following way:
\begin{equation}\label{reg_parameters_sparse}
\begin{split}
\lambda_1=C(\sigma\vee\mathbf{a})\sqrt{\dfrac{\mu_1\,\log(d)}{Nm}}\qquad\text{and} \qquad
\lambda_2=C(\sigma\vee\mathbf{a})\dfrac{\log(d)}{N}.
\end{split}
\end{equation}
With this choice of the regularization parameters, Corollary \ref{corollary_sparse_1} and Lemma \ref{lemma_stochastique-sparse} imply the following result.

\begin{Corollary}\label{upper_bound_sparse}
Let the distribution $\Pi$ on $\mathcal{X}'$ satisfy Assumptions  \ref{assPi}, and \ref{marginal_sparse}. Let  Assumption \ref{noise} hold and $\left\Vert L_{0}\right\Vert_{\infty}\leq \mathbf{a}$, $\left\Vert S_{0}\right\Vert_{\infty}\leq \mathbf{a}$. Assume that $N\leq m_1m_2$ and that condition \eqref{cond_n_spars} holds.
Then, with probability at least $1-6/d$ for any solution $(\hat L,\hat S)$ of the convex program \eqref{estimator_sparse} with the regularization parameters $(\lambda_1,\lambda_2)$ given by \eqref{reg_parameters_sparse}, we have
\begin{equation} \label{upper_bound_sparse_1}
\begin{split}
\dfrac{\Vert  L_0-\hat L\Vert^{2}_2}{ m_1m_2}+\dfrac{\Vert S_0-\hat S\Vert^{2}_2}{m_1m_2}&\leq C_{\mu,\mu_1}\ae(\sigma\vee\mathbf{a})^{2}\log (d)\dfrac{r\,M+\vert\tilde \Omega\vert}{n}+\dfrac{\mathbf{a}^{2}s}{m_1m_2}
\end{split}
\end{equation}
where $C_{\mu,\mu_1}>0$ can depend only on $\mu$ and $\mu_1$. Moreover, with the same probability
\begin{equation*}
\begin{split}
\dfrac{\Vert \hat S_{\mI}\Vert^{2}_2}{\vert\mI\vert}\leq C_{\mu,\mu_1}\dfrac{\ae(\sigma\vee\mathbf{a})^{2}\vert\tilde \Omega\vert\,\log (d)}{n}+\dfrac{\mathbf{a}^{2}s}{m_1m_2}.
\end{split}
\end{equation*}
\end{Corollary}
 \textbf{Remarks.}
\textbf{1.} As in the columnwise sparsity case, we can recognize two terms in the upper bound \eqref{upper_bound_sparse_1}. The first term is proportional to 
  $rM/n$. It is
   of the same order as the rate of convergence for the usual matrix completion, see \cite{KLT11,klopp_general}. The second term accounts for the corruptions and is proportional to the number $s$ of nonzero entries in $S_0$ and to the number of corrupted observations $\vert\tilde \Omega\vert$. We will prove in Section~\ref{lower_bounds} below that these error terms are of the correct order up to a logarithmic factor.
   
 \textbf{2.}  If $s\ll n<m_1m_2$, the bound~\eqref{upper_bound_sparse_1} implies that one can estimate a low-rank matrix from a nearly minimal number of observations, even when a part of the observations has been corrupted.
  
   \textbf{3.} If all entries of $A_0$ are observed, i.e., the matrix decomposition problem is considered, the bound \eqref{upper_bound_sparse_1} is analogous to the corresponding bound in  \cite{AgarwalNegahbanWainwright}. Indeed, then $\vert \tilde \Omega\vert\leq s, N=m_1m_2$, $\ae\leq 2$ and we get
    \begin{equation*} 
     \begin{split}
     \dfrac{\Vert L_0-\hat L\Vert_2^{2}}{m_1m_2}+\dfrac{\Vert  S_0-\hat S\Vert_2^{2}}{m_1m_2}&\lesssim\ (\sigma\vee\mathbf{a})^{2}\left (\dfrac{r\,M}{m_1m_2}+\dfrac{s}{m_1m_2}\right ).
     \end{split}
     \end{equation*}


\section{Minimax lower bounds}\label{lower_bounds}

In this section, we prove the minimax lower bounds showing that the
rates attained by  our estimator are optimal up to a logarithmic
factor.
We will denote by $\inf_{(\hat{L},\hat{S})}$ the infimum over all pairs of estimators
$(\hat{L},\hat{S})$ for the components $L_0$ and $S_0$ in the decomposition $A_0 = L_0 + S_0$ where both $\hat{L}$ and $\hat{S}$ take values in $\mathbb R^{m_1\times m_2}$. For any $A_0\in\R^{m_1\times m_2}$, let $\mathbb P_{A_0}$ denote the probability
distribution of the observations $(X_1,Y_1,\dots,X_n,Y_n)$ satisfying \eqref{model}.
%

We begin with the case of columnwise sparsity.
For any matrix $S \in \mathbb R^{m_1\times m_2}$, we denote by $\|S\|_{2,0}$ the number of nonzero columns of $S$. 
For any integers $0\leq r\le
\min(m_1,m_2)$, $0\leq s\leq m_2$ and any ${\mathbf a}>0$, we consider the class of matrices
\begin{equation}\label{Alb_gs}
 \begin{split}
 {\cal A}_{GS}(r,s,{\mathbf a})
 &= \left \{ A_0 = L_0 + S_0\in\,\mathbb R^{m_1\times m_2}:\,
 \mathrm{rank}(L_0)\leq r,\,\Vert L_0\Vert_{\infty}\leq {\mathbf a}, \right.\\
 &\hskip 4 cm\left. \text{and}\;\Vert S_0\Vert_{2,0}\,\leq\, s\,,\Vert S_0\Vert_{\infty}\,\leq\, {\mathbf a}\,\, \right \}
 \end{split}
 \end{equation}
  and  
define
$$
\psi_{GS}(N,r,s) = (\sigma
\wedge {\mathbf a})^2 \left (\frac{Mr+\vert\tilde \Omega\vert}{n}+\frac{s}{m_2}\right ).
$$
The following theorem gives a lower bound on the estimation risk in the case of columnwise sparsity.

\begin{theorem}\label{th:lower_group_sparse} Suppose that $m_1,m_2\geq 2$.  Fix ${\mathbf a}>0$ and integers
$1\leq  r\leq \min(m_1,m_2)$ and $1\leq s\leq m_2/2$. Let Assumption \ref{marginal_sparse} be satisfied. Assume that $Mr\leq n$, $\vert \tilde \Omega\vert\leq sm_1$ and $\ae\leq 1+s/m_2$. Suppose that
the variables $\xi_i$ are i.i.d. Gaussian
${\cal N}(0,\sigma^2)$, $\sigma^2>0$, for $i=1,\dots,n$.  Then, there
exist absolute constants $\beta\in(0,1)$ and $c>0$, such that
\begin{equation*}
\inf_{(\hat{L},\hat S)}
\sup_{\substack{(L_0,S_0)\in\,{\cal A}_{GS}( r,s,{\mathbf a})
}}
\mathbb P_{A_0}\left (\dfrac{\Vert \hat L-L_0\Vert_2^{2}}{m_1m_2}+\dfrac{\Vert \hat S-S_0\Vert_2^{2}}{m_1m_2} > c\psi_{GS}(N,r,s) \right )\ \geq\ \beta.
\end{equation*}
\end{theorem}

We turn now to the case of entrywise sparsity. For any matrix $S \in \mathbb R^{m_1\times m_2}$, we denote by $\|S\|_{0}$ the number of nonzero entries of $S$. For any integers $0\leq r\le
\min(m_1,m_2)$, $0\leq s \leq m_1m_2/2$ and any ${\mathbf a}>0$, we consider the class of matrices
\begin{align*}\label{Alb_s}
{\cal A}_{S}(r,s,a)&= \left \{A_0 = L_0 + S_0\in\,\mathbb R^{m_1\times m_2}:\,
\mathrm{rank}(L_0)\leq r,\, \right.\\
&\hskip 2 cm\left. \|S_0\|_{0} \leq s  ,\, \Vert L_0\Vert_{\infty}\,\leq {\mathbf a}, \Vert S_0\Vert_{\infty}\,\leq\, {\mathbf a}\right \}
\end{align*}
and define
$$
\psi_{S}(N,r,s) = (\sigma
\wedge {\mathbf a})^2 \left \{\frac{Mr + \vert\tilde \Omega\vert}{n}+\dfrac{s}{m_1m_2}\right \}.
$$
We have the following theorem for the lower bound in the case of entrywise sparsity.

\begin{theorem}\label{th:lower_sparse} Assume that $m_1,m_2\geq 2$. Fix ${\mathbf a}>0$ and integers
$1\leq  r\leq \min(m_1,m_2)$ and $1\leq s \leq m_1 m_2/2$. Let Assumption \ref{marginal_sparse} be satisfied. Assume that $Mr\leq n$ and there exists a constant $\rho>0$ such that $\vert \tilde \Omega\vert\leq \rho \,r\,M$.
Suppose that the variables $\xi_i$ are i.i.d. Gaussian
${\cal N}(0,\sigma^2)$, $\sigma^2>0$, for $i=1,\dots,n$.  Then, there
exist absolute constants $\beta\in(0,1)$ and $c>0$, such that
\begin{equation}\label{eq:lower1_s}
\inf_{(\hat{L},\hat S)}
\sup_{\substack{(L_0,S_0)\in\,{\cal A}_{S}(r,s,{\mathbf a})
}}
\mathbb P_{A_0}\bigg(\frac{\|\hat{L}-L_0\|^2_{2}}{m_1m_2} + \frac{\|\hat{S}-S_0\|^2_{2}}{m_1m_2} > c\psi_{S}(N,r,s) \bigg)\ \geq\ \beta.
\end{equation}
\end{theorem}

\bigskip


\appendix
\noindent {\bf \Large Appendix}
\section{Proofs of Theorem \ref{thm2} and of Corollary \ref{upper_bound_column}}\label{proof-thm2}
\subsection{Proof of Theorem \ref{thm2}}
The proofs of the upper bounds have similarities with the methods developed in \cite{klopp_general} for noisy matrix completion but the presence of corruptions in our setting requires a new approach, in particular, for proving ''restricted strong convexity property'' (Lemma \ref{constraind_S}) which is the main difficulty in the proof.

Recall that our estimator is defined as
\begin{equation*}
(\hat{L},\hat S)\in\underset{^{\left\Vert L\right\Vert_{\infty}\leq \mathbf{a}}_{\left\Vert S\right\Vert_{\infty}\leq \mathbf{a}}
}{\argmin}\left \{\dfrac{1}{N}\sum^{N}_{i=1} \left (Y_i-\left\langle X_i,L+S\right\rangle\right )^{2}+\lambda_{1} \Vert L\Vert_* +\lambda_{2}\mathcal{R}(S)\right \}
\end{equation*}
and our goal is to bound from above the Frobenius norms $\Vert  L_0-\hat L\Vert_2^{2}$ and $\Vert S_0-\hat S\Vert_2^{2}$.

\textbf{1)} Set 
 $\mathcal F(L,S)=\frac{1}{N}\sum^{N}_{i=1} \left (Y_i-\left\langle X_i,L+S\right\rangle\right )^{2}+\lambda_1 \Vert L\Vert_*+\lambda_2 \mathcal R( S)$, 
$\Delta L=L_0-\hat L$ and $\Delta S=S_0-\hat S$. Using  the inequality $\mathcal F(\hat L,\hat S)\leq \mathcal F(L_0,S_0)$ and \eqref{model} we get
\begin{equation*}
\dfrac{1}{N}\sum^{N}_{i=1} \left (\left\langle X_i,\Delta L+\Delta S\right\rangle+\xi_{i}\right )^{2}+\lambda_{1} \Vert \hat L\Vert_*+\lambda_{2}\mathcal{R}(\hat S)\leq \dfrac{1}{N}\sum^{N}_{i=1} \xi_{i}^{2}+\lambda_{1} \Vert L_0\Vert_*+\lambda_{2}\mathcal{R}(S_0).
\end{equation*}
After 
some algebra this implies 
\begin{equation}\label{rob_1}
\begin{split}
\dfrac{1}{N}\sum_{i\in \Omega} \left\langle X_i,\Delta L+\Delta S\right\rangle^{2}&\leq \underset{\mathbf {I}}{\underbrace{\dfrac{2}{N}\sum_{i\in\tilde\Omega} \left \vert\left\langle \xi_iX_i,\Delta L+\Delta S\right\rangle\right \vert-\dfrac{1}{N}\sum_{i\in \tilde\Omega} \left\langle X_i,\Delta L+\Delta S\right\rangle^{2}}}\\&\hskip 0.5 cm+ \underset{\mathbf {II}}{\underbrace{2\left \vert\left\langle \Sigma,\Delta L\right\rangle\right \vert+\lambda_{1}\left ( \Vert L_0\Vert_* -\Vert \hat L\Vert_*\right )}}\\&\hskip 1cm+\underset{\mathbf {III}}{\underbrace{2\left \vert\left\langle \Sigma,\Delta S_{\mI}\right\rangle\right \vert+\lambda_{2}\left ( \mR(S_0)-\mR(\hat S)\right )}}
\end{split}
\end{equation}
where $\Sigma=\frac{1}{N}\sum_{i\in \Omega} \xi_iX_i$ and we have used the equality  $\left\langle \Sigma,\Delta S\right\rangle=\left\langle \Sigma,\Delta S_{\mI}\right\rangle$.  We now  estimate each of the three terms on the right hand side of \eqref{rob_1} separately. This will be done on the random event 
\begin{equation}\label{random_event}
\mathcal{U}=\left \{\underset{1\leq i\leq N}{\max}\vert\xi_{i}\vert
\leq C_*\sigma\sqrt{\log d}\right \}
\end{equation}
where $C_*>0$ is a suitably chosen constant.
Using a standard bound on the maximum of sub-gaussian variables and the constraint $N\leq m_1m_2$ we get that there exists an absolute constant $C_*>0$  such that $\mathbb P(\mathcal{U})\geq 1-\frac{1}{2d}$. In what follows, we take this constant $C_*$ in the definition of $\mathcal{U}$.

We start by estimating $\mathbf{I}$. On the event $\mathcal{U}$, we get 
\begin{equation}\label{rob_4}
\begin{split}
{\bf I}&\leq \frac{1}{N}\sum_{i\in\tilde\Omega}  \xi^{2}_i \leq \dfrac{C\,\sigma^{2}\vert \tilde\Omega\vert\log(d)}{N}.
\end{split}
\end{equation}

Now we estimate $\mathbf{II}$. For a linear vector subspace $S$ of a euclidean space, let $P_S$ denote the orthogonal projector on $S$ and
 let $S^\bot$ denote the orthogonal complement of $S$.
 For any $A\in \mathbb{R}^{m_1\times m_2}$, let $u_j(A)$ and $v_j(A)$ be the left and right orthonormal singular vectors of $A$, respectively . Denote by $S_1(A)$ the linear span of $\{u_j(A)\}$, and by $S_2(A)$ the linear span of $\{v_j(A)\}$.
 We set
  \begin{equation*} 
  \begin{split}
  \mathbf P_A^{\bot}(B)=P_{S_1^{\bot}(A)}BP_{S_2^{\bot}(A)}\quad\text{and}\quad\mathbf P_A(B)=B- \mathbf P_A^{\bot}(B).
  \end{split}
  \end{equation*}
By definition of $\mathbf P_{L_0}^{\bot}$, for any matrix $B$ the singular vectors of $\mathbf P_{L_0}^{\bot}(B)$  are orthogonal to the space spanned by the singular vectors of $L_0$. This implies that $\left\Vert L_0+\mathbf P_{L_0}^{\bot}(\Delta L) \right\Vert_*=\left\Vert L_0 \right\Vert_*+\left\Vert \mathbf P_{L_0}^{\bot}(\Delta L) \right\Vert_*$. 
Thus,
\begin{equation*} \label{ineq}
\begin{split}
\| \hat L\|_*&=\left\Vert L_0 +\Delta L\right\Vert_*\\
&=\left\Vert L_0 +\mathbf P_{L_0}^{\bot}(\Delta L)+\mathbf P_{L_0}(\Delta L)\right\Vert_*\\
&\geq \left\Vert L_0 +\mathbf P_{L_0}^{\bot}(\Delta L)\right\Vert_*-\left\Vert\mathbf P_{L_0}(\Delta L)\right\Vert_*\\
&=\left\Vert L_0 \right\Vert_*+\left\Vert \mathbf P_{L_0}^{\bot}(\Delta L)\right\Vert_*-\left\Vert\mathbf P_{L_0}(\Delta L)\right\Vert_*,
\end{split}
\end{equation*}
which yields 
\begin{equation}\label{un2_correction}
\| L_0 \|_*-\| \hat L\|_*\leq \left\Vert\mathbf P_{L_0}(\Delta L)\right\Vert_*-\left\Vert \mathbf P_{L_0}^{\bot}(\Delta L)\right\Vert_*.
\end{equation}
 Using \eqref{un2_correction} and the duality between the nuclear and the operator norms, we obtain
\begin{equation*}
\begin{split}
\mathbf {II}\leq 2 \Vert \Sigma\Vert\Vert\Delta L\Vert_{*}+\lambda_{1}\left (\left\Vert\mathbf P_{L_0}(\Delta L)\right\Vert_*-\left \Vert \mathbf P_{L_0}^{\bot}(\Delta L)\right\Vert_*\right ).
\end{split}
\end{equation*}
The assumption that $\lambda_1\geq 4\Vert\Sigma\Vert$ and the triangle inequality imply 
\begin{equation}\label{rob_2}
\mathbf {II} \leq \dfrac{3}{2}\lambda_1\left\Vert\mathbf P_{L_0}(\Delta L)\right\Vert_*\leq \dfrac{3}{2}\lambda_1\sqrt{2r}\left\Vert\Delta L\right\Vert_2
\end{equation}
where $r=\rank (L_0)$ and we have used that $\rank(\mathbf P_{L_0}(\Delta L))\leq 2\,\rank(L_0)$.

For the third term in \eqref{rob_1}, we use the duality between the $\mR$ and  $\mR^{*}$, and the identity $\Delta S_{\mI}=-\hat S_{\mI}$:
\begin{equation*}
\begin{split}
\mathbf {III}\leq 2\mR^{*}(\Sigma)\mR(\hat S_{\mI})+\lambda_{2}\left ( \mR(S_0)-\mR(\hat S)\right ).
\end{split}
\end{equation*}
This and the assumption that $\lambda_2\geq 4\mR^{*}(\Sigma)$ imply
\begin{equation}\label{rob_3}
\mathbf {III}\leq \lambda_2\mR( S_0).
\end{equation}
Plugging \eqref{rob_2}, \eqref{rob_3} and \eqref{rob_4} in \eqref{rob_1} we get that, on the event $\mathcal{U}$,
\begin{equation} \label{rob_12}
\begin{split}
\dfrac{1}{n}\sum_{i\in\Omega} \left\langle X_i,\Delta L+\Delta S\right\rangle^{2}&\leq \frac{3\,\ae\,\lambda_1}{\sqrt{2}}\sqrt{r}\left\Vert\Delta L\right\Vert_2+\ae\lambda_2\mR( S_0) +\dfrac{C\sigma^{2}\vert \tilde\Omega\vert\log(d)}{n}
\end{split}
\end{equation}
where $\ae=N/n$.
\vskip 0.5 cm

\textbf{2)}  Second, we will show that a kind of restricted strong convexity 
holds for the random sampling operator given by $(X_i)$ on a suitable subset of matrices. In words, we prove that the observation operator captures a substantial component of any pair of matrices $L,S$ belonging  to a properly chosen {\it constrained set} (cf. Lemma \ref{constraind_S}(ii) below for the exact statement).  This will imply that, with high probability,
\begin{equation}\label{rob_5}
\dfrac{1}{n}\sum_{i\in\Omega} \left\langle X_i,\Delta L+\Delta S\right\rangle^{2}\geq \left\Vert \Delta L+\Delta S\right\Vert_{L_2(\Pi)}^{2}-\mathcal{E}
\end{equation}
with an appropriate residual $\mathcal{E}$, whenever we prove that $(\Delta L,\Delta S)$ belongs to the constrained set.  This will be a substantial element of the remaining part of the proof. The result of the theorem will then be deduced by combining \eqref{rob_12} and \eqref{rob_5}.

We start by defining our constrained set.
For positive constants $\delta_1$ and $\delta_2$,  we first introduce the following set of matrices where $\Delta S$ should lie:
 \begin{equation} \label{def_B}
 \mathcal{B}(\delta_1,\delta_2)=\left \{B\in\mathbb{R}^{m_1\times m_2} \,:\,  \left\Vert B\right\Vert_{L_2(\Pi)}^{2}\leq \delta^{2}_1\;\text{and}\; \mR(B)\leq \delta_2\right \}.
 \end{equation}
 The constants $\delta_1$ and $\delta_2$ define the constraints on the $L_2(\Pi)$-norm and on the sparsity of the component $S$. The error term $\mathcal{E}$ in \eqref{rob_5} depends on $\delta_1$ and $\delta_2$. We will specify the suitable values of  $\delta_1$ and $\delta_2$ for the matrix $\Delta S$ later. Next, we define the following set of pairs of matrices: 
 \begin{equation*} 
\begin{split}
&\mathcal{D}(\tau,\kappa)=\left \{(A,B)\in\bmR \,:\, \left\Vert A+B\right\Vert_{L_2(\Pi)}^{2}\geq \sqrt{\dfrac{64\,\log(d)}{\log\left (6/5\right )\,n}},\right .\\&\left . \hskip 5cm  \left\Vert A+B\right\Vert_{\infty}\leq1,\left\Vert A\right\Vert_{*}\leq \sqrt{\tau} \left\Vert A_{\mI}\right\Vert_{2}+\kappa\right \}
\end{split}
\end{equation*}
where $\kappa$ and $\tau<m_1\wedge m_2$ are some positive constants. This will be used for $A=\Delta L$ and $B=\Delta S$. If the $L_2(\Pi)$-norm of the sum of two matrices is too small, the right hand side of \eqref{rob_5} is negative. The first inequality in the  definition of $\mathcal{D}(\tau,\kappa)$ prevents from this. Condition $\left\Vert A\right\Vert_{*}\leq \sqrt{\tau} \left\Vert A_{\mI}\right\Vert_{2}+\kappa$ is a relaxed form of  the condition  $\left\Vert A\right\Vert_{*}\leq \sqrt{\tau} \left\Vert A\right\Vert_{2}$ satisfied by matrices with rank $\tau$. We will show that, with high probability, the matrix $\Delta L$ satisfies this condition with $\tau=C\,\rank (L_0)$ and a small $\kappa$. To prove it, we need the bound $\mR(B)\leq \delta_2$ on the corrupted part.  
  
Finally, define our {\it constrained set}  as the intersection   
$$ \mathcal{D}(\tau,\kappa)\cap \left \{\mathbb{R}^{m_1\times m_2}\times  \mathcal{B}(\delta_1,\delta_2)\right \}. $$

We now return to the proof of the theorem. To prove \eqref{upper_bound}, we bound separately the norms  $ \left\Vert \Delta L\right\Vert_2$ and $ \left\Vert \Delta S\right\Vert_2$. Note that
\begin{equation}\label{xx}
\begin{split}
\Vert \Delta L\Vert_2^{2}&\leq \Vert \Delta L_{\mI}\Vert_2^{2}+\Vert \Delta L_{\tilde\mI}\Vert_2^{2}
\leq \Vert \Delta L_{\mI}\Vert_2^{2}+4\mathbf a^{2}\vert \tilde \mI\vert\\
&\leq \mu  \vert \mI\vert \Vert \Delta L_{\mI}\Vert_{L_2(\Pi)}^{2}+4\mathbf a^{2}\vert \tilde \mI\vert
\end{split}
\end{equation} 
and similarly,
\begin{equation*}\label{}
\begin{split}
\Vert \Delta S\Vert_2^{2}
&\leq \mu  \vert \mI\vert \Vert \Delta S_{\mI}\Vert_{L_2(\Pi)}^{2}+4\mathbf a^{2}\vert \tilde \mI\vert.
\end{split}
\end{equation*} 
In view of these inequalities,  it is enough to bound the quantities $\Vert \Delta S_{\mI}\Vert_{L_2(\Pi)}^{2}$ and $\Vert \Delta L_{\mI}\Vert_2^{2}$. A bound on $\Vert \Delta S_{\mI}\Vert_{L_2(\Pi)}^{2}$ with the rate as claimed in \eqref{upper_bound} is given in Lemma \ref{bound_delta_S}  below. 
 In order to bound $\Vert \Delta L_{\mI}\Vert_{L_2(\Pi)}^{2}$ (or $\Vert \Delta L_{\mI}\Vert_2^{2}$ according to cases), we will need the following argument.


\textbf{Case 1}: Suppose that $\left\Vert \Delta L+\Delta S\right\Vert_{L_2(\Pi)}^{2} < 16\mathbf{a}^{2}\sqrt{\dfrac{64\,\log(d)}{\log\left (6/5\right )\,n}}$.  Then a straightforward inequality 
\begin{equation}\label{triangle}
\Vert \Delta L+\Delta S\Vert^{2}_{L_2(\Pi)}\geq \frac{1}{2} \Vert \Delta L\Vert^{2}_{L_2(\Pi)}-\Vert \Delta S\Vert^{2}_{L_2(\Pi)}
\end{equation}
together with Lemma \ref{bound_delta_S} below implies that, with probability at least $1-2.5/d$,
\begin{equation}\label{main_3}
\left\Vert \Delta L\right\Vert^{2}_{L_2(\Pi)}\leq \Delta_{1}
\end{equation}
where 
\begin{equation*}
\begin{split}
\Delta_1&=C\Psi_4/\mu = C\left \{\mathbf{a}^{2}\,\sqrt{\dfrac{\log(d)}{n}}+\mathbf{a}\,\mR(\mathbf{Id}_{\tilde{\Omega}})\left [\ae\,\lambda_{2}+ \mathbf a\,\bE\left ( \mR^{*}(\Sigma_R)\right )\right ]\right .\\&\hskip 2.5 cm\left .+\left (\dfrac{\mathbf a\,\bE\left (\mR^{*}(\Sigma_R)\right )}{\lambda_2}+\ae\right )\frac{\vert \tilde\Omega\vert\left (\mathbf a^{2}+\sigma^{2}\log(d)\right  )}{N}
\right \}.
\end{split}
\end{equation*}
Note also that $\Psi_4\le C(\Psi_1+\Psi_2 +\Psi_3)$.
In view of \eqref{xx}, \eqref{main_3} and of fact that $\vert \mI\vert\le m_1m_2$, the bound on $\Vert\Delta L\Vert_2^{2}$ stated in the theorem holds with probability at least $1-2.5/d$.

\textbf{Case 2}: Assume now that  $\left\Vert \Delta L+\Delta S\right\Vert_{L_2(\Pi)}^{2} \geq 16\mathbf{a}^{2}\sqrt{\dfrac{64\,\log(d)}{\log\left (6/5\right )\,n}}$.  We will show that in this case and with an appropriate choice of $\delta_1,\delta_2,\tau$ and $\kappa$, the pair $\frac{1}{4\mathbf{a}}(\Delta L, \Delta S)$ belongs to the  intersection   $ \mathcal{D}(\tau,\kappa)\cap \left \{\mathbb{R}^{m_1\times m_2}\times  \mathcal{B}(\delta_1,\delta_2)\right \}$.

Lemma \ref{nuclear_bound} below and \eqref{rob_4} imply that, on the event $\mathcal{U}$,
\begin{equation}\label{condition_L}
\begin{split}
\Vert\Delta L\Vert_{*}&\leq 4\sqrt{2r}\Vert\Delta L\Vert_{2}+\frac{\lambda_2\,\mathbf{a}}{\lambda_1}\,\mR(\mathbf{Id}_{\tilde{\Omega}})+\dfrac{C\sigma^{2}\vert \tilde\Omega\vert\log(d)}{N\lambda_1}\\
&\leq 4\sqrt{2r}\Vert\Delta L_{\mI}\Vert_{2}+8\mathbf{a}\sqrt{2r\vert\tilde{\mI}\vert}+\frac{\lambda_2\,\mathbf{a}}{\lambda_1}\,\mR(\mathbf{Id}_{\tilde{\Omega}})+\dfrac{C\sigma^{2}\vert \tilde\Omega\vert\log(d)}{N\lambda_1}.
\end{split}
\end{equation}
Lemma \ref{bound_delta_S} yields that, with probability at least $1-2.5\,d^{-1}$, 
$$\dfrac{\Delta S}{4\mathbf{a}}\in \mathcal{B}\left (\dfrac{\sqrt{\Delta_1}}{4\mathbf{a}},2\mR(\mathbf{Id}_{\tilde{\Omega}})+\frac{C\vert \tilde\Omega\vert\left (\mathbf a^{2}+\sigma^{2}\log(d)\right  )}{4\mathbf{a}N\lambda_2}\right )=\bar{\mathcal{B}}.$$
This property and \eqref{condition_L} imply that 
$ \frac{1}{4\mathbf{a}}\left (\Delta L,\Delta S\right )\in \mathcal{D}(\tau,\kappa)\cap \left \{\mathbb{R}^{m_1\times m_2}\times  \bar{\mathcal{B}}\right \}$ with probability at least $1-2.5\,d^{-1}$, where 
 $$\tau=32r\quad\text{and}\quad\kappa=2\sqrt{2r\vert\tilde{\mI}\vert}+\frac{\lambda_2}{4\lambda_1}\,\mR(\mathbf{Id}_{\tilde{\Omega}})+\dfrac{C\sigma^{2}\vert \tilde\Omega\vert\log(d)}{4\mathbf{a}\,N\lambda_1}.$$
Therefore, we can apply Lemma \ref{constraind_S}(ii).
From Lemma \ref{constraind_S}(ii)
   and \eqref{rob_12} we obtain that, 
with probability at least $1-4.5\,d^{-1}$, 
\begin{equation}  \label{rob_6}
\begin{split}
\dfrac{1}{2}\Vert \Delta L+\Delta S\Vert^{2}_{L_2(\Pi)}&\leq\frac{3\,\ae\,\lambda_1}{\sqrt{2}}\sqrt{r}\left\Vert\Delta L\right\Vert_2 
+C\mathcal{E}
\end{split}
\end{equation}
where 
\begin{equation}
\begin{split}
\mathcal{E}&=\mu\,\mathbf{a}^{2}\,r\, \vert \mI\vert \left (\bE\left ( \Vert\Sigma_R\Vert\right )\right )^{2}+8\mathbf{a}^{2}\sqrt{2r\vert\tilde{\mI}\vert}\,\bE\left ( \Vert\Sigma_R\Vert\right )+\lambda_2\mR(\mathbf{Id}_{\tilde{\Omega}})\left (\dfrac{\mathbf{a}^{2}\bE\left ( \Vert\Sigma_R\Vert\right )}{\lambda_1}+\mathbf{a}\,\ae\right )\\&\hskip 1 cm + \dfrac{\vert \tilde\Omega\vert
\left (\mathbf a^{2}+\sigma^{2}\log(d)\right  )}{N}\left (\dfrac{\mathbf a\,\bE\left ( \Vert\Sigma_R\Vert\right )}{\lambda_1}+\dfrac{\mathbf a\,\bE \left (\mR^{*}(\Sigma_R)\right )}{\lambda_2}+\ae\right )+\Delta_1.
\end{split}
\end{equation}
Using an elementary argument and then \eqref{xx} we find
\begin{equation*}
\begin{split}
\frac{3\,\ae}{\sqrt{2}}\lambda_1\sqrt{r}\left\Vert\Delta L\right\Vert_2&\leq \dfrac{9\,\ae^{2}\,\mu\,m_1\,m_2\,r\,\lambda_1^{2}}{2}+\dfrac{\left\Vert\Delta L\right\Vert^{2}_2}{4\mu m_1 m_2}\\
&\leq \dfrac{9\,\ae^{2}\,\mu\,m_1\,m_2\,r\,\lambda_1^{2}}{2}+\dfrac{\left\Vert\Delta L_{\mI}\right\Vert^{2}_2}{4\mu m_1 m_2}+\dfrac{\mathbf{a}^{2}\vert \tilde{\mI}\vert}{\mu\,m_1m_2}.
\end{split} 
\end{equation*}
This inequality and \eqref{rob_6} yield
\begin{equation*} 
\begin{split}
\Vert \Delta L+\Delta S\Vert^{2}_{L_2(\Pi)}&\leq \dfrac{9\,\ae^{2}\,\mu\,m_1\,m_2\,r\,\lambda_1^{2}}{4}+\dfrac{\left\Vert\Delta L_{\mI}\right\Vert^{2}_2}{4\mu m_1 m_2}+\dfrac{\mathbf{a}^{2}\vert \tilde{\mI}\vert}{\mu\,m_1m_2} +C\mathcal{E}.
\end{split}
\end{equation*}
Using again \eqref{triangle}, Lemma \ref{bound_delta_S}, \eqref{ass1} and  the bound $\vert \mI\vert\leq m_1m_2$  we obtain
\begin{equation*} 
\begin{split}
\dfrac{\Vert \Delta L_{\mI}\Vert^{2}_2}{\mu m_1 m_2}&\leq C\left \{ \ae^{2}\,\mu\,m_1\,m_2\,r\,\lambda_1^{2}+\dfrac{\mathbf{a}^{2}\vert \tilde{\mI}\vert}{\mu\,m_1m_2} +\mathcal{E}\right \}.
\end{split}
\end{equation*}
This and the inequality $\sqrt{2r\vert\tilde{\mI}\vert}\,\bE\left ( \Vert\Sigma_R\Vert\right )\leq  \dfrac{\vert \tilde{\mI}\vert}{\mu\,m_1m_2}+\mu\,m_1m_2\,r\,\left ( \bE\left ( \Vert\Sigma_R\Vert\right )\right )^{2}$ imply that, with probability at least $1-4.5\,d^{-1}$,
\begin{equation} \label{main_2}
\begin{split}
\dfrac{\Vert \Delta L_{\mI}\Vert_2^{2}}{m_1 m_2}&\leq C\left \{ \Psi_1+ \Psi_2+\Psi_3\right \}\,.
\end{split}
\end{equation}
In view of 
  \eqref{main_2} and \eqref{xx}, $\Vert \Delta L\Vert_2^{2}$ is bounded by the right hand side of \eqref{upper_bound} with probability at least $1-4.5\,d^{-1}$. Finally, inequality \eqref{upper_bound1} follows from Lemma \ref{bound_delta_S},  \eqref{ass1} and the identity $\Delta S_{\mI}=-\hat S_{\mI}$.

\begin{lemma}\label{R_bound}
Assume that $\lambda_2\geq 4\left (\mR^{*}(\Sigma)+2\mathbf{a}\mR^{*}(W)\right )$. Then, we have
\begin{equation*}
\mR(\Delta S_{\mI})\leq 3\mR(\Delta S_{\tilde\Omega})+\frac{1}{N\lambda_2}\left [4\mathbf a^{2}\vert \tilde\Omega\vert+\sum_{i\in\tilde\Omega} \xi_i^{2}\right ]
\end{equation*}
\end{lemma}
\begin{proof}
Let $\partial \Vert \cdot\Vert_{*}$, and $\partial \mathcal{R}$ denote the subdifferentials of $\Vert \cdot\Vert_{*}$ and of $\mathcal{R}$, respectively. By the standard condition for optimality over a convex set (see \cite{aubin}, Chapter 4,
Section 2, Corollary 6), we have
\begin{equation}\label{optimality}
\begin{split}
-\dfrac{2}{N}\sum^{N}_{i=1} \left (Y_i-\left\langle X_i,\hat L+\hat S\right\rangle\right )&\left\langle X_i,L+S-\hat L-\hat S\right\rangle\\&+\lambda_{1}\left\langle \partial \Vert \hat L\Vert_{*},L-\hat L\right\rangle+\lambda_{2}\left\langle \partial \mathcal{R}(\hat S),S-\hat S\right\rangle\geq 0
\end{split}
\end{equation}
for all feasible pairs $(L,S)$.
In particular, for $(\hat L,S_0)$ we obtain 
\begin{equation*}
\begin{split}
-\dfrac{2}{N}\sum^{N}_{i=1} \left (Y_i-\left\langle X_i,\hat L+\hat S\right\rangle\right )&\left\langle X_i,\Delta S\right\rangle+\lambda_{2}\left\langle \partial \mathcal{R}(\hat S),\Delta S\right\rangle\geq 0,
\end{split}
\end{equation*}
which implies
\begin{equation*}
\begin{split}
&-\dfrac{2}{N}\sum^{N}_{i=1} \left\langle X_i,\Delta S\right\rangle^{2}-\dfrac{2}{N}\sum_{i\in\tilde\Omega} \left\langle X_i,\Delta L\right\rangle\left\langle X_i,\Delta S\right\rangle-\dfrac{2}{N}\sum_{i\in\tilde\Omega} \xi_i\left\langle  X_i,\Delta S\right\rangle\\&\hskip 0.5 cm-\dfrac{2}{N}\sum_{i\in\Omega} \left\langle X_i,\Delta L\right\rangle\left\langle X_i,\Delta S\right\rangle-2\left\langle \Sigma,\Delta S\right\rangle+\lambda_{2}\left\langle \partial \mathcal{R}(\hat S),\Delta S\right\rangle\geq 0.
\end{split}
\end{equation*}
Using the elementary inequality $2ab\le a^2+b^2$ and the bound
$\Vert \Delta L\Vert_{\infty}\leq 2\mathbf a$ we find 
\begin{equation*}
\begin{split}
&-\dfrac{2}{N}\sum^{N}_{i=1} \left\langle X_i,\Delta S\right\rangle^{2}-\dfrac{2}{N}\sum_{i\in\tilde\Omega} \left\langle X_i,\Delta L\right\rangle\left\langle X_i,\Delta S\right\rangle-\dfrac{2}{N}\sum_{i\in\tilde\Omega} \xi_i\left\langle  X_i,\Delta S\right\rangle\\&\hskip 2.5 cm\leq \dfrac{1}{N}\sum_{i\in\tilde\Omega} \left\langle X_i,\Delta L\right\rangle^{2}+\dfrac{1}{N}\sum_{i\in\tilde\Omega}\xi^{2}\\&\hskip 3cm\leq \dfrac{4\mathbf a^{2}\vert \tilde\Omega\vert}{N}+\frac{1}{N}\sum_{i\in\tilde\Omega}\xi_i^{2}.
\end{split}
\end{equation*}
Combining the last two displays we get
\begin{equation}\label{robust_6}
\begin{split}
\lambda_{2}\left\langle \partial \mathcal{R}(\hat S),\hat S-S_0\right\rangle&\leq 2\left \vert\left\langle \dfrac{1}{N}\sum_{i\in \Omega} \left\langle X_i,\Delta L\right\rangle X_i,\Delta S\right\rangle\right \vert+2\left \vert\left\langle \Sigma,\Delta S\right\rangle\right \vert+\dfrac{4\mathbf a^{2}\vert \tilde\Omega\vert}{N}+\frac{1}{N}\sum_{i\in\tilde\Omega}\xi_i^{2}\\&
\leq2\mR^{*}\left (\dfrac{1}{N}\sum_{i\in\Omega} \left\langle X_i,\Delta L\right\rangle X_i\right )\mR(\Delta S)+2\mR^{*}( \Sigma)\mR(\Delta S)\\
&\hskip 0.5 cm+\dfrac{4\mathbf a^{2}\vert \tilde\Omega\vert}{N}+\frac{1}{N}\sum_{i\in\tilde\Omega}\xi_i^{2}.
\end{split}
\end{equation}
By Lemma \ref{lem-44},
\begin{equation}\label{robust_5}
\begin{split}
\mR^{*}\left (\frac{1}{N}\sum_{i\in \Omega} \left\langle X_i,\Delta L\right\rangle X_i\right )\leq2\mathbf{a}\mR^{*}(W)
\end{split}
\end{equation}
where $W=\frac{1}{N}\sum_{i\in \Omega}X_{i}$. On the other hand, the convexity of $\mR(\cdot)$ and the definition of subdifferential imply
\begin{equation}\label{robust_7}
\begin{split}
\mR(S_0)\geq \mR(\hat S)+\left\langle \partial \mathcal{R}(\hat S),\Delta S\right\rangle.
\end{split}
\end{equation}
Plugging \eqref{robust_5} and \eqref{robust_7} in \eqref{robust_6} we obtain
\begin{equation*}
\begin{split}
\lambda_{2}\left(\mR(\hat S)-\mR(S_0)\right)\leq 4\mathbf{a}\mR^{*}(W)\mR(\Delta S)+2\mR^{*}( \Sigma)\mR(\Delta S)+\dfrac{4\mathbf a^{2}\vert \tilde\Omega\vert}{N}+\frac{1}{N}\sum_{i\in\tilde\Omega}\xi_i^{2}.
\end{split}
\end{equation*}
Next, the decomposability of $\mR(\cdot)$, the identity $(S_0)_{\mI}=0$ and the triangle inequality yield
\begin{equation*}\label{robust_R}
\begin{split}
\mR(S_0-\Delta S)-\mR(S_0)&=\mR\left ((S_0-\Delta S)_{\tilde\mI}\right )+\mR\left ((S_0-\Delta S)_{\mI}\right )-\mR\left ((S_0)_{\tilde\mI}\right )\\& \geq\mR\left ((\Delta S)_{\mI}\right )-\mR\left ((\Delta S)_{\tilde\mI}\right ).
\end{split}
\end{equation*}
Since $\lambda_2\geq 4\left (2\mathbf{a}\mR^{*}(W)+\mR^{*}( \Sigma)\right )$ the last two displays imply
\begin{equation*}
\begin{split}
&\lambda_2\left (\mR\left ((\Delta S)_{\mI}\right )-\mR\left ((\Delta S)_{\tilde\mI}\right )\right )\\&\hskip 2cm\leq \dfrac{\lambda_2}{2}\left (\mR\left (\Delta S_{\tilde \mI}\right )+\mR\left ((\Delta S)_{\mI}\right )\right )+\dfrac{4\mathbf a^{2}\vert \tilde\Omega\vert}{N}+\frac{1}{N}\sum_{i\in\tilde\Omega}\xi_i^{2}.
\end{split}
\end{equation*}
Thus,
\begin{equation}\label{julio_1}
\begin{split}
&\mR\left (\Delta S_{\mI}\right )\leq 3\mR\left (\Delta S_{\tilde\mI}\right )+\frac{1}{N\lambda_2}\left [4\mathbf a^{2}\vert \tilde\Omega\vert+\sum_{i\in\tilde\Omega} \xi_i^{2}\right ].
\end{split}
\end{equation}
Since we assume that all unobserved entries of $S_0$  are zero, we have $(S_0)_{\tilde{\mI}}=(S_0)_{\tilde{\Omega}}$. On the other hand, $S_{\tilde {\mI}}=\hat S_{\tilde {\Omega}}$ as  $\mR(\cdot)$ is a monotonic norm. Indeed, 
adding to $S$ a non-zero element on the non-observed part increases $\mR(S)$ but does not modify $\frac{1}{N}\sum^{N}_{i=1} \left (Y_i-\left\langle X_i,L+S\right\rangle\right )^{2}$. To conclude, we have $\Delta S_{\tilde\mI}=\Delta S_{\tilde\Omega}$,
which together with \eqref{julio_1}, implies the  Lemma.
\end{proof}

\begin{lemma}\label{nuclear_bound}
Suppose that $\lambda_1\geq 4\Vert\Sigma\Vert$ and $\lambda_2\geq 4\mR^{*}(\Sigma)$. Then,
\begin{equation*}
\left\Vert \mathbf P_{L_0}^{\bot}(\Delta L)\right\Vert_*\leq3\left\Vert\mathbf P_{L_0}(\Delta L)\right\Vert_*+\frac{\lambda_2\,\mathbf{a}}{\lambda_1}\,\mR(\mathbf{Id}_{\tilde{\Omega}})+\frac{1}{N\lambda_1}\sum_{i\in\tilde\Omega} \xi_i^{2}.
\end{equation*}
\end{lemma}
\begin{proof}
Using \eqref{optimality} for $(L,S)=(L_0,S_0)$ we obtain
\begin{equation}\label{lemma2_1}
\begin{split}
&-\dfrac{2}{N}\sum^{N}_{i=1}\left\langle X_i,\Delta S+\Delta L\right\rangle^{2} -\dfrac{2}{N}\sum_{i\in\tilde\Omega} \left\langle \xi_iX_i,\Delta L+\Delta S\right\rangle\\&\hskip 0.5 cm-2\left\langle \Sigma,\left (\Delta S\right )_{\mI}\right\rangle-2\left\langle \Sigma,\Delta L\right\rangle+\lambda_{1}\left\langle \partial \Vert \hat L\Vert_{*},\Delta L\right\rangle+\lambda_{2}\left\langle \partial \mathcal{R}(\hat S),\Delta S\right\rangle\geq 0.
\end{split}
\end{equation}
The convexity of $\Vert\cdot\Vert_{*}$ and of $\mR(\cdot)$ and the definition of the subdifferential imply  
\begin{equation*}
\begin{split}
\Vert L_0\Vert_{*}&\geq \Vert\hat L\Vert_{*}+\left\langle \partial \Vert\hat L\Vert_{*},\Delta L\right\rangle\\
\mR(S_0)&\geq \mR(\hat S)+\left\langle \partial \mathcal{R}(\hat S),\Delta S\right\rangle.
\end{split}
\end{equation*}
Together with \eqref{lemma2_1}, this yields
\begin{equation*}
\begin{split}
\lambda_{1}\left (\Vert\hat L\Vert_{*}-\Vert L_0\Vert_{*}\right )+\lambda_{2}\left(\mR(\hat S)-\mR(S_0)\right)&\leq 2\Vert \Sigma \Vert\Vert \Delta L\Vert_{*}+2\mR^{*}( \Sigma)\mR\left (\Delta S_{\mI}\right )\\&\hskip 1 cm+\frac{1}{N}\sum_{i\in\tilde\Omega} \xi_i^{2}.
\end{split}
\end{equation*}
Using the conditions $\lambda_1\geq 4\Vert\Sigma\Vert$, $\lambda_2\geq 4\mR^{*}(\Sigma)$, the triangle inequality and \eqref{un2_correction}  we get 
\begin{equation*}
\begin{split}
&\lambda_1\left (\left\Vert \mathbf P_{L_0}^{\bot}(\Delta L)\right\Vert_*- \left\Vert\mathbf P_{L_0}(\Delta L)\right\Vert_*\right )+\lambda_2\left (\mR(\hat S)-\mR(S_0)\right )
\\&\hskip 0.5 cm\leq \dfrac{\lambda_1}{2}\left (\left\Vert \mathbf P_{L_0}^{\bot}(\Delta L)\right\Vert_*
+ \left\Vert\mathbf P_{L_0}(\Delta L)\right\Vert_* \right)+\dfrac{\lambda_2}{2}\mR\left (\hat S_{\mI}\right )+\frac{1}{N}\sum_{i\in\tilde\Omega} \xi_i^{2}.
\end{split}
\end{equation*}
Since we assume that all unobserved entries of $S_0$  are zero, we obtain  $\mR(S_0)\leq \mathbf{a}\mR(\mathbf{Id}_{\tilde{\Omega}})$. Using this inequality in the last display proves the lemma.
\end{proof}
\begin{lemma}\label{bound_delta_S}
Let $n>m_1$ and $\lambda_2\geq 4\left ( \mR^{*}(\Sigma)+2\mathbf{a}\mR^{*}(W)\right )$. Suppose that the distribution $\Pi$ on $\mathcal{X}'$ satisfies Assumptions \ref{assPi} and \ref{L}. Let  $ \left\Vert S_0\right\Vert_{\infty}\leq \mathbf{a}$ for some constant $\mathbf{a}$ and let Assumption \ref{noise} be satisfied.
 Then, with probability at least $1-2.5\,d^{-1}$,
\begin{equation} \label{3x}
\begin{split}
\Vert \Delta S\Vert_{L_2(\Pi)}^{2}&\leq  C \Psi_4/\mu,
\end{split}
\end{equation}
and
\begin{equation}\label{condition_S}
\begin{split}
\mR(\Delta S)&\leq 8\mathbf{a}\mR(\mathbf{Id}_{\tilde{\Omega}})+\frac{\vert \tilde\Omega\vert\left (4\mathbf a^{2}+C\sigma^{2}\log(d)\right  )}{N\lambda_2}.
\end{split}
\end{equation}

\end{lemma}
\begin{proof}
Using  the inequality $\mathcal F(\hat L,\hat S)\leq \mathcal F(\hat L,S_0)$ and \eqref{model} we obtain
\begin{equation*}
\begin{split}
&\dfrac{1}{N}\sum^{N}_{i=1} \left (\left\langle X_i,\Delta L+\Delta S\right\rangle+\xi_{i}\right )^{2}+\lambda_{2}\mathcal{R}(\hat S)\\&\hskip 2 cm\leq \dfrac{1}{N}\sum^{N}_{i=1} \left (\left\langle X_i,\Delta L\right\rangle+\xi_{i}\right )^{2}+\lambda_{2}\mathcal{R}(S_0)
\end{split}
\end{equation*}
which implies
\begin{equation*}
\begin{split}
\dfrac{1}{N}\sum_{i\in \Omega} \left\langle X_i,\Delta S\right\rangle^{2}&+\dfrac{1}{N}\sum_{i\in\tilde\Omega} \left\langle X_i,\Delta S\right\rangle^{2}+\dfrac{2}{N}\sum_{i\in\tilde{\Omega}} \left\langle X_i,\Delta L\right\rangle\left\langle X_i,\Delta S\right\rangle+\dfrac{2}{N}\sum_{i\in\tilde{\Omega}} \left\langle \xi_iX_i,\Delta S\right\rangle\\&\hskip 1 cm+\dfrac{2}{N}\sum_{i\in\Omega} \left\langle X_i,\Delta L\right\rangle\left\langle X_i,\Delta S_{\mI}\right\rangle+2\left\langle \Sigma,\Delta S_{\mI}\right\rangle+\lambda_{2} \mathcal{R}(\hat S)\leq \lambda_{2} \mathcal{R}(S_0).
\end{split}
\end{equation*}
From Lemma \ref{lem-44} and  the duality between $\mR$ and $\mR^{*}$ we obtain
\begin{equation*}
\begin{split}
\dfrac{1}{N}\sum_{i\in\Omega} \left\langle X_i,\Delta S\right\rangle^{2}&\leq 2\left (2\mathbf{a}\,\mR^{*}(W)+\mR^{*}( \Sigma)\right )\mR(\Delta S_{\mI})+ \lambda_{2} \left (\mathcal{R}(S_0)-\mathcal{R}(\hat S)\right )\\&\hskip 2 cm+ \dfrac{2}{N}\sum_{i\in\tilde\Omega} \left\langle X_i,\Delta L\right\rangle^{2}+\dfrac{2}{N}\sum_{i\in\tilde\Omega}\xi^{2}.
\end{split}
\end{equation*}
Since here $\Delta S_{\mI}=-\hat S_{\mI}$ and $\lambda_2\geq 4\left (\mR^{*}(\Sigma)+2\mathbf{a}\mR^{*}(W)\right )$ it follows that
\begin{equation}\label{lemma11_1}
\begin{split}
\dfrac{1}{N}\sum_{i\in\Omega} \left\langle X_i,\Delta S\right\rangle^{2}\leq  \lambda_{2}\mR\left  ( S_0\right )+\dfrac{2}{N}\sum_{i\in\tilde\Omega} \left\langle X_i,\Delta L\right\rangle^{2}+\dfrac{2}{N}\sum_{i\in\tilde\Omega}\xi^{2}.
\end{split}
\end{equation}
Now, Lemma \ref{R_bound} and the bound $\Vert \Delta S\Vert_{\infty}\leq 2\mathbf{a}$ imply that, on the event $\mathcal{U}$ defined in \eqref{random_event},
\begin{equation}\label{lemma_S_1}
\begin{split}
\mR(\Delta S)&\leq 4\mR(\Delta S_{\tilde\Omega})+\frac{\vert \tilde\Omega\vert\left (4\mathbf a^{2}+C\sigma^{2}\log(d)\right  )}{N\lambda_2}\\&\leq 
8\mathbf{a}\mR(\mathbf{Id}_{\tilde{\Omega}})+\frac{\vert \tilde\Omega\vert\left (4\mathbf a^{2}+C\sigma^{2}\log(d)\right  )}{N\lambda_2}.
\end{split}
\end{equation}
 Thus, \eqref{condition_S} is proved. To prove \eqref{3x}, consider the following two cases. 
 
 \textbf{Case I:} $\Vert \Delta S\Vert^{2}_{L_{2}(\Pi)}<  4\mathbf a^{2}\sqrt{\frac{64\log(d)}{\log (6/5)\,n}}$.  Then  \eqref{3x} holds trivially. 
 
 \textbf{Case II:} $\Vert \Delta S\Vert^{2}_{L_{2}(\Pi)}\ge 4\mathbf a^{2}\sqrt{\frac{64\log(d)}{\log (6/5)\,n}}$. Then inequality \eqref{lemma_S_1} and the bound $\Vert \Delta S\Vert_{\infty}\leq 2\mathbf{a}$ imply that, on the event $\mathcal{U}$,
 $$\dfrac{\Delta S}{2\mathbf a}\in \mathcal{C}\left (4\,\mR(\mathbf{Id}_{\tilde{\Omega}})+\,\frac{\vert \tilde\Omega\vert\left (8\mathbf a^{2}+C\sigma^{2}\log(d)\right  )}{2\mathbf a\,N\lambda_2}\right )
 $$ 
 where, for any $\delta>0$, the set $\mathcal{C}(\delta)$ is defined as:
 \begin{equation}\label{set_constrained_S}
 \mathcal{C}(\delta)=\left \{A\in\mathbb{R}^{m_1\times m_2}:\left\Vert A\right\Vert_{\infty}\leq 1, \left\Vert A\right\Vert_{L_2(\Pi)}^{2}\geq \sqrt{\dfrac{64\,\log(d)}{\log\left (6/5\right )\,n}}, \mR(A)\leq \delta\right \}.
 \end{equation}
 Thus, we can apply Lemma \ref{constraind_S}(i) below. In view of this lemma,  the inequalities \eqref{lemma11_1}, \eqref{rob_4}, $\Vert \Delta L\Vert_{\infty}\leq 2\mathbf{a}$ and $\mR\left  (S_0\right )\leq \mathbf{a}\mR(\mathbf{Id}_{\tilde{\mI}})$ imply that \eqref{3x} holds with probability at least $1-2.5\,d^{-1}$.
\end{proof}

\begin{lemma}\label{constraind_S}
Let the distribution $\Pi$ on $\mathcal{X}'$ satisfy Assumptions \ref{assPi} and \ref{L}. Let $\delta, \delta_1, \delta_2, \tau$, and $\kappa$ be positive constants. Then, the following properties hold.
\begin{itemize}
\item[(i)] With probability at least $1-\dfrac{2}{d}$, 
\begin{equation*}
\frac{1}{n}\sum_{i\in \Omega} \left\langle X_i,S\right\rangle^{2}\geq \dfrac{1}{2}\Vert S\Vert _{L_2(\Pi)}^{2}-8\delta\bE\left ( \mR^{*}(\Sigma_R)\right )
\end{equation*}
for any  $S\in \mathcal{C}(\delta)$.
\item[(ii)] With probability at least $1-\dfrac{2}{d}$, 
\begin{equation*}
\begin{split}
\frac{1}{n}\sum_{i\in \Omega} \left\langle X_i,L+S\right\rangle^{2}&\geq \dfrac{1}{2}\Vert L+S\Vert _{L_2(\Pi)}^{2}-
\left \{360\mu\,\vert \mI\vert\,\tau\left ( \bE\left ( \Vert\Sigma_R\Vert\right )\right )^{2}\right .\\&\hskip 0.5 cm\left .+4\delta_1^{2}+8\delta_2\,\bE\left ( \mR^{*}(\Sigma_R)\right )+8\kappa\bE\left ( \Vert\Sigma_R\Vert\right )\right \}
\end{split}
\end{equation*}
for any pair $(L,S)\in \mathcal{D}(\tau,\kappa)\cap \left \{\mathbb{R}^{m_1\times m_2}\times  \mathcal{B}(\delta_1,\delta_2)\right \}$.
\end{itemize}
\end{lemma}
\begin{proof}
We give a unified proof of (i) and (ii). Let $A=S$ for (i) and $A=L+S$ for (ii). Set
\begin{equation*}
\mathcal{E}= \left\{
     \begin{array}{lll}
                  8\delta\bE\left ( \mR^{*}(\Sigma_R)\right ) 
    & \mbox{for (i)}  \\
  360\mu\,\vert \mI\vert\,\tau\left ( \bE\left ( \Vert\Sigma_R\Vert\right )\right )^{2}+4\delta_1^{2} +8\delta_2\,\bE\left ( \mR^{*}(\Sigma_R)\right )+8\kappa\bE\left ( \Vert\Sigma_R\Vert\right )
       & \mbox{\text{for (ii)}} 
         \end{array} \right.
 \end{equation*}
 and
 \begin{equation*}
 \mathcal{C}= \left\{
      \begin{array}{lll}
                   \mathcal{C}(\delta)
     & \mbox{for (i)}  \\
     \mathcal{D}(\tau,\kappa)\cap \left (\mathbb{R}^{m_1\times m_2}\times  \mathcal{B}(\delta_1,\delta_2)\right ) 
        & \mbox{\text{for (ii).}} 
          \end{array} \right.
  \end{equation*}
To prove the lemma it is enough to show
 that the probability of the random event 
\begin{equation*}
\mathcal{B}=\left \{\exists\,A\in \mathcal{C}\,\text{such that}\,\left \vert\frac{1}{n}\sum_{i\in \Omega} \left\langle X_i,A\right\rangle^{2}-\Vert A\Vert _{L_2(\Pi)}^{2}\right \vert> \dfrac{1}{2}\Vert A\Vert _{L_2(\Pi)}^{2}+ \mathcal{E}\right \}
\end{equation*}
is smaller than $2/d$.
In order to estimate the probability of $\mathcal{B}$, we use a standard peeling argument. Set $\nu=\sqrt{\dfrac{64\,\log(d)}{\log\left (6/5\right )\,n}}$ and $\alpha=\dfrac{6}{5}$. For $l\in\mathbb N$, define $$S_l=\left \{A\in \mathcal{C}\,:\,\alpha^{l-1}\nu \leq \Vert A\Vert _{L_2(\Pi)}^{2}\leq \alpha^{l}\nu\right \}.$$ 
If the event $\mathcal{B}$ holds, there exist $l\in\mathbb N$ and a matrix $A\in \mathcal{C}\cap S_l$ such that 
\begin{equation}\label{Bl}
\begin{split}
\left \vert\frac{1}{n}\sum_{i\in \Omega} \left\langle X_i,A\right\rangle^{2}-\Vert A\Vert _{L_2(\Pi)}^{2}\right \vert&> \dfrac{1}{2}\Vert A\Vert _{L_2(\Pi)}^{2}+ \mathcal{E}\\&> \dfrac{1}{2}\alpha^{l-1}\nu+ \mathcal{E}\\&
= \dfrac{5}{12}\alpha^{l}\nu+ \mathcal{E}.
\end{split}
\end{equation}
For each $l\in\mathbb N$, consider the random event 
$$\mathcal{B}_l=\left \{\exists\,A\in \mathcal{C}'(\alpha^{l}\nu)\,:\,\left \vert\frac{1}{n}\sum_{i\in \Omega} \left\langle X_i,A\right\rangle^{2}-\Vert A\Vert _{L_2(\Pi)}^{2}\right \vert> \dfrac{5}{12}\alpha^{l}\nu+ \mathcal{E}\right \}$$
where 
$$\mathcal{C}'(T)=\left \{A\in\mathcal{C} \,:\,  \left\Vert A\right\Vert_{L_2(\Pi)}^{2}\leq T \right \}, \quad \forall T>0.
$$ 
Note that $A\in S_l$ implies that $A\in \mathcal{C}'(\alpha^{l}\nu)$. This and \eqref{Bl} grant the inclusion $\mathcal{B}\subset\cup_{l=1}^\infty \,\mathcal{B}_l$. By Lemma \ref{Z_T},
$\mathbb P\left (\mathcal{B}_l\right )\leq \exp(-c_5\,n\,\alpha^{2l}\nu^{2})$ where $c_5=1/128$. Using the union bound we find
\begin{equation*}
\begin{split}
\mathbb P\left (\mathcal{B}\right )&\leq \underset{l=1}{\overset{\infty}{\Sigma}}\mathbb P\left (\mathcal{B}_l\right )\\&\leq \underset{l=1}{\overset{\infty}{\Sigma}}\exp(-c_5\,n\,\alpha^{2l}\,\nu^{2})\\&\leq \underset{l=1}{\overset{\infty}{\Sigma}}\exp\left (-\left (2\,c_5\,n\,\log(\alpha)\,\nu^{2}\right )l\right )
\end{split}
\end{equation*} 
where we have used the inequality $e^{x}\geq x$. We finally obtain, for $\nu=\sqrt{\dfrac{64\,\log(d)}{\log\left (6/5\right )\,n}}$,
\begin{equation*}
\mathbb P\left (\mathcal{B}\right )\leq \dfrac{\exp\left (-2\,c_5\,n\,\log(\alpha)\,\nu^{2}\right )}{1-\exp\left (-2\,c_5\,n\,\log(\alpha)\,\nu^{2}\right )}=\dfrac{\exp\left (-\log(d)\right )}{1-\exp\left (-\log(d)\right )}.
\end{equation*} 
\end{proof}
Let
$$Z_T=\underset{A\in \mathcal{C}'(T)}{\sup}\left \vert\frac{1}{n}\sum_{i\in \Omega} \left\langle X_i,A\right\rangle^{2}-\Vert A\Vert _{L_2(\Pi)}^{2}\right \vert.$$ 
\begin{lemma}\label{Z_T}
Let the distribution $\Pi$ on $\mathcal{X}'$ satisfy Assumptions \ref{assPi} and \ref{L}.
Then, 
  $$\mathbb P\left (Z_T>\dfrac{5}{12}T+\mathcal{E}\right )\leq \exp(-c_5\,n\,T^{2})$$
where $c_5=\dfrac{1}{128}$.
 
\end{lemma}
\begin{proof}
We follow a standard approach: first we show that $Z_T$ concentrates around its expectation and then we bound from above the expectation. Since $\left\Vert A\right\Vert_{\infty}\le1$ for all $A\in \mathcal{C}'(T)$, we have $\left \vert\left\langle X_i,A\right\rangle\right \vert\leq 1$.   We use first a Talagrand type 
concentration inequality, cf. \cite[Theorem 14.2]{van_de_geer}, implying that
\begin{equation}\label{concentration}
\mathbb P\left  (Z_T\geq \bE \left ( Z_T\right )+\dfrac{1}{9}\left (\dfrac{5}{12}T\right )\right )\leq \exp\left (-c_5\,n\,T^{2}\right )
\end{equation}
where $c_5=\dfrac{1}{128}$.
Next, we bound the expectation $\bE\left ( Z_T\right )$. By a standard symmetrization argument (see e.g. \cite[Theorem 2.1]{koltchiskii-stflour}) we obtain 
\begin{equation*}
\begin{split}
\bE \left ( Z_T\right )&= \bE\left (\underset{A\in \mathcal{C}'(T)}{\sup}\left \vert\frac{1}{n}\sum_{i\in \Omega} \left\langle X_i,A\right\rangle^{2}-\bE\left (\left\langle X,A\right\rangle ^{2}\right )\right \vert\right )\\&\leq 2\bE\left (\underset{A\in \mathcal{C}'(T)}{\sup}\left \vert\dfrac{1}{n} \sum_{i\in \Omega} \epsilon_i\left\langle X_i,A\right\rangle ^{2} \right \vert\right )
\end{split}\end{equation*}
where $\{\epsilon_i\}_{i=1}^{n}$ is an i.i.d. Rademacher sequence. Then, the contraction inequality (see e.g. \cite{koltchiskii-stflour}) yields
\begin{equation*}
\bE \left ( Z_T\right )\leq 8\bE\left (\underset{A\in \mathcal{C}'(T)}{\sup}\left \vert\dfrac{1}{n} \sum_{i\in \Omega} \epsilon_i\left\langle X_i,A\right\rangle\right \vert\right )=8\bE\left (\underset{A\in \mathcal{C}'(T)}{\sup}\left \vert \left\langle \Sigma_R,A\right\rangle\right \vert\right )
\end{equation*}
where $\Sigma_R=\dfrac{1}{n} \Sum \epsilon_i X_i$.
Now, to obtain a bound on
 $\bE\left (\underset{A\in \mathcal{C}'(T)}{\sup}\left \vert \left\langle \Sigma_R,A\right\rangle\right \vert\right )$ we will consider separately the cases $ \mathcal{C}=\mathcal{C}(\delta)$ and $\mathcal{C}= \mathcal{D}(\tau,\kappa)\cap \left \{\mathbb{R}^{m_1\times m_2}\times  \mathcal{B}(\delta_1,\delta_2)\right \}$. 
 
 \textbf{Case I:} $A\in \mathcal{C}(\delta)$ and $\left\Vert A\right\Vert_{L_2(\Pi)}^{2}\leq T$. 
 By the definition of $\mathcal{C}(\delta)$ we have 
$\mR( A) \leq \delta.$
 Thus, by the duality between $\mR$ and $\mR^{*}$,  
\begin{equation*}
\bE \left ( Z_T\right )\leq 8\bE\left (\underset{\mR( A)\leq \delta }{\sup}\left \vert \left\langle \Sigma_R,A\right\rangle\right \vert\right )\leq 8\delta\,\bE\left ( \mR^{*}(\Sigma_R)\right ).
\end{equation*}
This and
 the concentration inequality \eqref{concentration} imply
$$\mathbb P\left (Z_T>\dfrac{5}{12}T+ \mathcal{E}\right )\leq \exp(-c_5\,n\,T^{2})$$
with $c_5=\dfrac{1}{128}$ and $\mathcal{E}=8\delta\,\bE\left ( \mR^{*}(\Sigma_R)\right )$ as stated.

\textbf{Case II:} $A=L+S$ where $(L,S)\in \mathcal{D}(\tau,\kappa)$, $S\in \mathcal{B}(\delta_1,\delta_2)$, and $\Vert L+S\Vert^{2}_{L_2(\Pi)}\leq T$. Then, by the definition of $\mathcal{B}(\delta_1,\delta_2)$, we have 
$
\mR(S)\leq \delta_2.
$
On the other hand, the definition of $\mathcal{D}(\tau,\kappa)$ yields
\begin{equation*}
\Vert L\Vert_{*}\leq \sqrt{\tau}\Vert L_{\mI}\Vert_{2}+\kappa
\end{equation*}
and 
\begin{equation*}
\Vert L\Vert_{L_2(\Pi)}\leq \Vert L+S\Vert_{L_2(\Pi)}+\Vert S\Vert_{L_2(\Pi)}\leq \sqrt{T}+\delta_1.
\end{equation*}
The last two inequalities imply 
\begin{equation*}
\Vert L\Vert_{*}\leq \sqrt{\mu\,\vert \mI\vert\,\tau}(\sqrt{T}+\delta_1)+\kappa:=\Gamma_{1}.
\end{equation*}
Therefore we can write
\begin{equation*}
\begin{split}
\bE\left (\underset{A\in \mathcal{C}'(T)}{\sup}\left \vert \left\langle \Sigma_R,A\right\rangle\right \vert\right )
&\leq 
8\bE\left (\underset{\Vert L\Vert_{*}\leq \Gamma_{1} }{\sup}\left \vert \left\langle \Sigma_R,L\right\rangle\right \vert+\underset{\mR( S)\leq \delta_{2} }{\sup}\left \vert \left\langle \Sigma_R,S\right\rangle\right \vert\right )\\&\hskip 0.5 cm \leq 8\left \{\Gamma_{1}\,\bE\left ( \Vert\Sigma_R\Vert\right )+\delta_{2}\,\bE\left ( \mR^{*}(\Sigma_R)\right )\right \}.
\end{split}
\end{equation*}
Combining this bound with the following elementary inequalities: 
\begin{equation*}
\begin{split}
\frac{1}{9}\left (\frac{5}{12}T\right )+8\sqrt{\mu\,\vert \mI\vert\,\tau\,T}\,\bE\left ( \Vert\Sigma_R\Vert\right )
&\leq \left (\frac{1}{9}+\frac{8}{9}\right ) \frac{5}{12}T +44\mu\,\vert \mI\vert\,\tau\left ( \bE\left ( \Vert\Sigma_R\Vert\right )\right )^{2},\\
\delta_1\sqrt{\mu\,\vert \mI\vert\,\tau}_,\bE\left ( \Vert\Sigma_R\Vert\right )&\leq \mu\,\vert \mI\vert\,\tau\left ( \bE\left ( \Vert\Sigma_R\Vert\right )\right )^{2}+\frac{\delta_1^{2}}{2}
\end{split}
\end{equation*}
and using the concentration bound \eqref{concentration} we obtain $$\mathbb P\left (Z_T>\dfrac{5}{12}T+ \mathcal{E}\right )\leq \exp(-c_5\,n\,T^{2})$$
with $c_5=\dfrac{1}{128}$ and 
\begin{equation}
\begin{split}
\mathcal{E}&=360\mu\,\vert \mI\vert\,\tau\left ( \bE\left ( \Vert\Sigma_R\Vert\right )\right )^{2}+4\delta_1^{2}+8\delta_2\,\bE\left ( \mR^{*}(\Sigma_R)\right )+8\kappa\bE\left ( \Vert\Sigma_R\Vert\right )
\end{split}
\end{equation} as stated.
\end{proof}


\subsection{Proof of Corollary \ref{upper_bound_column}}
With $\lambda_1$ and $\lambda_2$ given by \eqref{reg_parameters_column} we obtain
\begin{equation*}
\begin{split}
\Psi_1&=\mu^{2}\ae^{2}(\sigma\vee\mathbf{a})^{2}\dfrac{M\,r\,\log d}{N},\\
\Psi'_2&\leq  \mu^{2}\ae^{2}(\sigma\vee\mathbf{a})^{2}\log(d)\dfrac{\vert\tilde \Omega\vert}{N}+\dfrac{\mathbf{a}^{2}s}{m_2},\\
\Psi'_3&=\dfrac{\mu\,\ae\vert \tilde\Omega\vert
\left (\mathbf a^{2}+\sigma^{2}\log(d)\right  )}{N}+\dfrac{\mathbf{a}^{2}s}{m_2}
\\
\Psi'_4&\leq\dfrac{\mu\,\ae^{2}\vert \tilde\Omega\vert
\left (\mathbf a^{2}+\sigma^{2}\log(d)\right  )}{N}+\mathbf{a}^{2}\,\sqrt{\dfrac{\log(d)}{n}}+\dfrac{\mathbf{a}^{2}s}{m_2}.
\end{split}
\end{equation*}
\section{Proof of Theorems \ref{th:lower_group_sparse} and \ref{th:lower_sparse}}

Note that the §assumption $\ae\leq 1+s/m_2$ implies that 
\begin{equation}\label{lower_condition}
\dfrac{\vert\tilde \Omega\vert}{n}\leq \dfrac{s}{m_2}.
\end{equation}
Assume w.l.o.g. that $m_1\geq m_2$. For a $\gamma\leq 1$, define
$$
\tilde{ \mathcal{L}} \, =\Big\{ \tilde{L} = (l_{ij})\in\R^{m_1\times r}:
l_{ij}\in\Big\{0, \gamma(\sigma \wedge \mathbf{a}) \Big(\frac{
 r M}{n}\Big)^{1/2}\Big\}\,, \forall 1\leq i \leq m_1,\, 1\leq j\leq
r\Big\},
$$
and consider the associated set of block matrices
$$
\mathcal{L} \ =\ \Big\{
L=(\begin{array}{c|c|c|c}\tilde{L}&\cdots&\tilde{L}&O
\end{array})\in\R^{m_1\times m_2}: \tilde{L}\in \tilde{\mathcal{L}}\Big\},
$$
where $O$ denotes the $m_1\times (m_2-r\lfloor m_2/(2r) \rfloor )$ zero
matrix, and $\lfloor x \rfloor$ is the integer part of $x$.

We define similarly the set of matrices
$$
\tilde {\mathcal{S}} \, =\Big\{ \tilde{S}=(s_{ij})\in\R^{m_1\times s}:
s_{ij}\in\big\{0, \gamma(\sigma \wedge \mathbf{a})\big\} \,, \forall 1\leq i \leq m_1,\, 1\leq j\leq
s\Big\},
$$
and
$$
\mathcal{S}\ =\ \Big\{
S=(\begin{array}{c|c}\tilde O&\tilde{S} 
\end{array})\in\R^{m_1\times m_2}: \tilde{S} \in \tilde {\mathcal{S}}\Big\},
$$
where $\tilde O$ is the $m_1\times (m_2-s )$ zero
matrix.
We now set 
$$
\mathcal A = \left\lbrace  A = L+S\,:\, L \in \mathcal L,\, S \in \mathcal S  \right\rbrace.
$$ 

\begin{remark}
In the case $m_1< m_2$, we only need to change the construction of the low rank component of the test set. We first introduce a matrix $\tilde L = \left(\begin{array}{c|c}\bar L&O\\\end{array}\right) \in \mathbb R^{r \times m_2} $ where $\bar L \in \mathbb \mathbb R^{r \times (m_2/2)}$ with entries in $\left\{0, \gamma (\sigma \wedge \mathbf{a}) \left(\frac{ rM}{n} \right)^{1/2}\right\}$ and then we replicate this matrix to obtain a block matrix $L$ of size $m_1 \times m_2$ 
$$
L=\left(
\begin{array}{c}
\tilde{L}\\
\hline\\
\vdots\\
\hline\\
\tilde{L}\\
\hline\\
O
\end{array}
\right).
$$
\end{remark}

By construction, any element of $\mathcal{A}$ as well
as the difference of any two elements of $\mathcal{A}$ can be decomposed into a low rank component $L$ of rank at most $r$ and a group sparse component $S$ with at most $s$ nonzero columns. In addition, the entries of any matrix in
$\mathcal{A}$ take values in $[0,a]$. Thus,
$\mathcal{A}\subset {\cal A}_{GS}(r,s,\mathbf{a})$. 

 We first establish a lower bound of the order $rM/n$. Let $\tilde {\mathcal{A}}\subset \mathcal{A}$ be such that for any $A=L+S\in \tilde {\mathcal{A}}$ we have $S=\mathbf{0}$. The Varshamov-Gilbert bound (cf. Lemma 2.9 in \cite{tsy_09}) guarantees the existence of a subset $\AA^0\subset\tilde{\mathcal{A}}$ with
cardinality $\mathrm{Card}(\AA^0) \geq 2^{(rM)/8}+1$ containing the
zero $m_1\times m_2$ matrix ${\bf 0}$ and such that, for any two
distinct elements $A_1$ and $A_2$ of $\AA^0$,
\begin{equation}\label{lower_2}
\Arrowvert A_1-A_2\Arrowvert_{2}^2  \geq \frac{Mr}{8}
\left(\gamma^2(\sigma \wedge \mathbf{a})^2 \frac{Mr }{n} \right)
\left\lfloor \frac{m_2}{r}\right\rfloor \geq
\frac{\gamma^2}{16}(\sigma \wedge \mathbf{a})^2\,m_1m_2 \,\frac{Mr}{n}\,.
\end{equation}

Since $\xi_i\sim {\cal N}(0,\sigma^2)$ we get that, for any $A\in \AA^0$, the Kullback-Leibler
divergence $K\big(\P_{{\bf 0}},\P_{A}\big)$ between $\P_{{\bf 0}}$
and $\P_{A}$ satisfies
\begin{equation}\label{KLdiv}
K\big(\P_{{\bf 0}},\P_{A}\big)\ =\
\frac{\vert\Omega\vert}{2\sigma^2}\|A\|_{L_2(\Pi)}^2 \leq
\frac{\mu_1\gamma^2\,Mr}{2}
\end{equation}
where we have used Assumption \ref{marginal_sparse}.
From (\ref{KLdiv}) we deduce that the condition
\begin{equation}\label{eq: condition C}
\frac{1}{\mathrm{Card}(\AA^0)-1} \sum_{A\in\AA^0}K(\P_{\bf
0},\P_{A})\ \leq\ \frac1{16} \log \big(\mathrm{Card}(\AA^0)-1\big)
\end{equation}
is satisfied  if $ \gamma>0$ is chosen as a
sufficiently small numerical constant.  In view
of (\ref{lower_2}) and (\ref{eq: condition C}),  the application of Theorem 2.5 in \cite{tsy_09} implies
\begin{equation}\label{lower_1}
\inf_{(\hat{L},\hat S)}
\sup_{\substack{(L_0,S_0)\in\,{\cal A}_{GS}( r,s,\mathbf{a})
}}
\mathbb P_{A_0}\left (\dfrac{\Vert \hat L-L_0\Vert_2^{2}}{m_1m_2}+\dfrac{\Vert \hat S-S_0\Vert_2^{2}}{m_1m_2} > \dfrac{C(\sigma \wedge \mathbf{a})^2\,Mr}{n} \right )\ \geq\ \beta
\end{equation}
for some  absolute constants $\beta\in(0,1)$.

We now prove the lower bound relative to the corruptions. Let $\bar {\mathcal{A}}\subset \mathcal{A}$ such that for any $A=L+S\in \bar {\mathcal{A}}$ we have $L=\mathbf{0}$. The Varshamov-Gilbert bound (cf. Lemma 2.9 in \cite{tsy_09}) guarantees the existence of a subset $\AA^0\subset\bar{\mathcal{A}}$ with
cardinality $\mathrm{Card}(\AA^0) \geq 2^{(sm_1)/8}+1$ containing the
zero $m_1\times m_2$ matrix ${\bf 0}$ and such that, for any two
distinct elements $A_1$ and $A_2$ of $\AA^0$,
\begin{equation}\label{lower_3}
\Arrowvert S_1-S_2\Arrowvert_{2}^2  \geq \frac{sm_1}{8}
\left(\gamma^2(\sigma \wedge \mathbf{a})^2 \right)
 =
\frac{\gamma^2\,(\sigma \wedge \mathbf{a})^2\,s}{8m_2}\,m_1m_2.
\end{equation}
For any $A\in \AA_0$, the Kullback-Leibler
divergence between $\P_{{\bf 0}}$
and $\P_{A}$ satisfies
\begin{equation*}
K\big(\P_{{\bf 0}},\P_{A}\big)\ =\
\frac{\vert\tilde\Omega\vert}{2\sigma^2}\gamma^2(\sigma \wedge \mathbf{a})^2\leq
\frac{\gamma^2\,m_1s}{2}
\end{equation*}
which implies that  condition
\eqref{eq: condition C} is satisfied if $ \gamma>0$ is chosen small enough. Thus, applying Theorem 2.5 in \cite{tsy_09} we get
\begin{equation}\label{lower_4}
\inf_{(\hat{L},\hat S)}
\sup_{\substack{(L_0,S_0)\in\,{\cal A}_{GS}( r,s,\mathbf{a})
}}
\mathbb P_{A_0}\left (\dfrac{\Vert \hat L-L_0\Vert_2^{2}}{m_1m_2}+\dfrac{\Vert \hat S-S_0\Vert_2^{2}}{m_1m_2} > \dfrac{C(\sigma \wedge \mathbf{a})^2\,s}{m_2} \right )\ \geq\ \beta
\end{equation}
for some  absolute constant $\beta\in(0,1)$. Theorem \ref{th:lower_group_sparse} follows from inequalities \eqref{lower_condition}, \eqref{lower_1} and \eqref{lower_4}.

 The proof of Theorem \ref{th:lower_sparse} follows the same lines as that of Theorem \ref{th:lower_group_sparse}. The only difference is that we replace  $\tilde{\mathcal{S}}$ by  the following set
 $$
 \Big\{ S=(s_{ij})\in\R^{m_1\times m_2}:
 s_{ij}\in\Big\{0, \gamma(\sigma \wedge \mathbf{a}) \Big\}\,, \forall 1\leq i \leq m_1,\, \lfloor m_2/2 \rfloor +1\leq j\leq
 m_2\Big\}.
 $$
 We omit further details here. 

\section{Proof of Lemma \ref{lemma_stocastique}}
Part (i) of Lemma \ref{lemma_stocastique} is proved in Lemmas 5  and  6 in \cite{klopp_general}.

 {\it Proof of (ii)}. For the sake of brevity, we set $X_i(j,k) = \langle X_i, e_j(m_1)e_k(m_2)^{\top}\rangle$.
By definition of $\Sigma$ and $\|\cdot\|_{2,\infty}$, we have
$$
\|\Sigma\|_{2,\infty}^2 = \max_{1\leq k \leq m_2}\sum_{j=1}^{m_1} \left(  \frac{1}{N}\sum_{i \in \Omega}\xi_{i}X_i(j,k) \right)^2.
$$
For any fixed $k$, we have
\begin{align}\label{proof_lemma_sto_interm1}
\sum_{j=1}^{m_1}\left(  \frac{1}{N}\sum_{i \in \Omega}\xi_{i}X_i(j,k) \right)^2 &= 
 \frac{1}{N^2}\sum_{i_1, i_2\in\Omega}\xi_{i_1}\xi_{i_2}\sum_{j= 1}^{m_1} X_{i_1}(j,k)X_{i_2}(j,k)\notag\\
&= \Xi^\top A_k \Xi,
\end{align}
where $\Xi = (\xi_1,\cdots,\xi_n)^\top$ and $A_k\in \mathbb R^{|\Omega|\times |\Omega|}$ with entries
$$
a_{i_1 i_2}(k) = \frac{1}{N^2}\sum_{j=1}^{m_1} X_{i_1}(j,k) X_{i_2}(j,k).
$$
We freeze the $X_i$ and we apply the version of Hanson--Wright inequality in \cite{RudelsonVershynin} to get  that there exists a numerical constant $C$ such that with probability at least $1-e^{-t}$
\begin{align}\label{Hanson-wright-column}
\left| \Xi^{\top} A_k \Xi - \mathbb E[\Xi^\top A_k \Xi \vert X_i] \right| \leq  C \sigma^2 \left(  \|A_k\|_2 \sqrt{t} + \|A_k\| t \right).
\end{align}
Next, we note that
\begin{align*}
\|A_k\|_2^2 = \sum_{i_1,i_2} a_{i_1 i_2}^2(k) &\leq \frac{1}{N^4} \sum_{i_1 i_2} \left( \sum_{j_1=1}^{m_1} X_{i_1}^2 (j_1,k)  \right) \left( \sum_{j_1=1}^{m_1} X_{i_2}^2 (j_1,k) \right)\\
&\leq  \frac{1}{N^4}\left[\sum_{i_1 }\sum_{j_1=1}^{m_1} X_{i_1}^2 (j_1,k) \right]^2= \left[\frac{1}{N^2}\sum_{i_1 }\sum_{j_1=1}^{m_1} X_{i_1} (j_1,k) \right]^2,
\end{align*}
where we have used the Cauchy - Schwarz inequality in the first line and the relation $X^{2}_{i} (j,k)=X_{i} (j,k)$.

Note that $Z_i(k):= \sum_{j= 1}^{m_1} X_i(j,k)$ follows a Bernoulli distribution with parameter $\pi_{\cdot k}$ and consequently $Z(k) = \sum_{i\in \Omega} Z_{i}(k)$ follows a Binomial distribution $ B(|\Omega|,\pi_{\cdot k})$. We apply Bernstein's inequality (see, e.g., \cite{van_de_geer}, page 486) to get that, for any $t>0$,
$$
\mathbb P \left( |Z(k) - \mathbb E[ Z(k)]| \geq 2\sqrt{  |\Omega| \pi_{\cdot k} t} + t  \right)  \leq 2 e^{-t}.
$$
Consequently, we get with probability at least $1-2e^{-t}$ that
$$
\|A_k\|_2^2 \leq \left(\frac{|\Omega| \pi_{\cdot k} + 2\sqrt{|\Omega|  \pi_{\cdot k} t} + t}{N^2}\right)^2
$$
and, using $\|A_k\|\leq \|A_k\|_{2}$, that
$$
\|A_k\| \leq \frac{|\Omega| \pi_{\cdot k} + 2\sqrt{|\Omega|  \pi_{\cdot k} t} + t}{N^2}.
$$
Note also that
$$
\mathbb E[\Xi^\top A_k \Xi \vert X_i]  = \frac{\sigma^2}{N^2}Z(k). 
$$
Combining the last three displays with (\ref{Hanson-wright-column}) we get, up to a rescaling of the constants, with probability at least $1-e^{-t}$ that
$$
\sum_{j=1}^{m_1}\left(  \frac{1}{N}\sum_{i \in \Omega_r}\xi_{i}X_i(j,k) \right)^2 \leq C \frac{\sigma^2}{N^2} \left( |\Omega| \pi_{\cdot k} + 2\sqrt{|\Omega| \pi_{\cdot k} t} + t\right)(1+\sqrt{t}+t).
$$
Replacing $t$ by $t+\log m_2$ in the above display and using the union bound gives that, with probability at least $1-e^{-t}$,
\begin{equation*}
\begin{split}
\|\Sigma\|_{2,\infty} \leq C \frac{\sigma}{N} &\left( |\Omega| \pi_{\cdot k} + 2\sqrt{|\Omega| \pi_{\cdot k} (t+\log m_2)} + (t+\log m_2)\right)^{1/2}\\&\times(1+\sqrt{t+\log m_2}+t+\log m_2)^{1/2}\\
&= C \frac{\sigma}{N}\left( \sqrt{|\Omega| \pi_{\cdot k}} +  \sqrt{t+\log m_2}\right) \left (1+\sqrt{t+\log m_2}\right ).
\end{split}
\end{equation*}
Assuming that $\log m_2 \geq 1$  we get with probability at least $1-e^{-t}$ that
$$
\|\Sigma\|_{2,\infty} \leq C \frac{\sigma}{N} \left( \sqrt{ |\Omega| \pi_{\cdot k} (t+\log m_2)} + (t+\log m_2)\right).
$$
 Using \eqref{milder-marginal}, we get  that there exists a numerical constant $C>0$ such with probability at least $1 -e^{-t}$
$$
\|\Sigma\|_{2,\infty} \leq C \frac{\sigma}{N} \left( \sqrt{ \dfrac{\gamma^{1/2}n(t+\log m_2)}{m_2}} + (t+\log m_2)\right).
$$
Finally,
we use Lemma \ref{lem-tech} to obtain the required bound on $\mathbb E\|\Sigma\|_{2,\infty}$.

{\it Proof of (iii).} We follow the same lines as in the proof of part (ii) above. The only difference is to replace $\xi_i$ by $\epsilon_i$, $\sigma$ by $1$ and $N$ by $n$.

{\it Proof of (iv).}
We need to establish the bound on 
$$
\|W\|_{2,\infty}^2 = \max_{1\leq k \leq m_2}\sum_{j=1}^{m_1} \left(  \frac{1}{N}\sum_{i \in \Omega}X_i(j,k) \right)^2.
$$
For any fixed $k$, we have
\begin{align*}\label{proof_lemma_sto_interm2}
\sum_{j=1}^{m_1}\left(  \frac{1}{N}\sum_{i \in \Omega}X_i(j,k) \right)^2 &= \frac{1}{N^2}\sum_{i \in \Omega}  \sum_{j= 1}^{m_1} X_i^2(j,k) + \frac{1}{N^2}\sum_{i_1\neq i_2}\sum_{j= 1}^{m_1} X_{i_1}(j,k)X_{i_2}(j,k).\notag\\
\end{align*}
The first term on the right hand side of the last display can be written as
$$
\frac{1}{N^2}\sum_{i \in \Omega}  \sum_{j= 1}^{m_1} X_i^2(j,k) = \frac{1}{N^2}\sum_{i \in \Omega}  \sum_{j= 1}^{m_1} X_i(j,k)=\frac{Z(k)}{N^2}.
$$
Using the concentration bound on $Z(k)$ in the proof of part (ii) above, we get that, with probability at least $1-e^{-t}$,
\begin{equation}\label{W-diag term}
\frac{1}{N^2}\sum_{i \in \Omega}  \sum_{j= 1}^{m_1} X_i^2(j,k) \leq \frac{|\Omega|}{N^2} \pi_{\cdot k} + 2 \frac{\sqrt{|\Omega| \pi_{\cdot k}t}}{N^2} + \frac{t}{N^2}.
\end{equation}
Next, the random variable 
$$
U_2 = \frac{1}{N^2}\sum_{i_1\neq i_2}\sum_{j= 1}^{m_1} \left [X_{i_1}(j,k)X_{i_2}(j,k) - \pi_{j,k}^2\right ]
$$
is a U-statistic of order 2. We use now a Bernstein-type concentration inequality for U-statistics. To this end, we set $X_i(\cdot,k) = (X_i(1,k),\cdots, X_i(m_1,k))^\top$ and $$h(X_{i_1}(\cdot,k),X_{i_2}(\cdot,k)) = \sum_{j= 1}^{m_1} \left [X_{i_1}(j,k)X_{i_2}(j,k) - \pi_{j,k}^2\right ].$$ 
Let $e_0(m_1) = \mathbf{0}_{m_1}$ be the zero vector in $\R^{m_1}$. Note that $ X_i(\cdot,k)$ takes values in $ \{ e_j(m_1),\, 0\leq j \leq m_1\}$.  For any function $g\,:\, \{ e_j(m_1),\, 0\leq j \leq m_1\}^2 \rightarrow \mathbb R$, we set $\|g\|_{L^{\infty}} = \max_{0\leq j_1,j_2\leq m_1}|g(e_{j_1}(m_1),e_{j_2}(m_1))|$.

We will need the following quantities to control the tail behavior of $U_2$
\begin{align*}
\mathbf A &= \|h\|_{L^\infty}, \quad \mathbf B^2 = \max\left\{  \left\| \sum_{i_1} \mathbb E h^2(X_{i_1}(\cdot,k),\cdot)  \right\|_{L^\infty}, \left\| \sum_{i_2} \mathbb E h^2(\cdot,X_{i_2}(\cdot,k))  \right\|_{L^\infty} \right\},\\
\mathbf C &= \sum_{i_1\neq i_2} \mathbb E \left [ h^2(X_{i_1}(\cdot,k),X_{i_2}'(\cdot,k))\right  ]\quad\text{and}\\
\mathbf D &= \sup\left\{ \mathbb E \sum_{i_1\neq i_2}  h\left [X_{i_1}(\cdot,k), X_{i_2}'(\cdot,k)\right ] f_{i_1}[X_{i_1}(\cdot,k)] g_{i_2}[X_{i_2}'(\cdot,k)],\right.\\
&\hspace{2cm}\left.\mathbb E \sum_{i_1} f_{i_1}^2(X_{i_1}(\cdot,k))\leq 1, \mathbb E \sum_{i_2} g_{i_2}^2(X_{i_2}'(\cdot,k))\leq 1  \right\},
\end{align*}
where $X_i'(\cdot,k)$ are independent replications of $X_i(\cdot,k)$ and  $f$, $g\,:\, \mathbb R^{m_1} \rightarrow \mathbb R$.

We now evaluate the above quantities in our particular setting. It is not hard to see that $\mathbf A  =\max\{ \pi_{\cdot k}^{(2)} \,,\,1 - \pi_{\cdot k}^{(2)} \}\leq 1$ where $\pi_{\cdot k}^{(2)}= \sum_{j=1}^{m_1}\pi_{jk}^2 $. We also have that
\begin{align*}
\mathbf C  &= \sum_{i_1\neq i_2} \left[\mathbb E \left [ \left \langle X_{i_1}(\cdot,k),X_{i_2}'(\cdot,k)\right \rangle^2 \right ] - \left( \sum_{j=1}^{m_1}\pi_{jk}^2 \right)^2\right]\\
&= |\Omega|(|\Omega| - 1)\left[\mathbb E \left [\left  \langle X_{i_1}(\cdot,k),X_{i_2}'(\cdot,k)\right \rangle \right ] - \left( \sum_{j=1}^{m_1}\pi_{jk}^2 \right)^2\right]\\
&= |\Omega|(|\Omega| - 1) \left [\sum_{j=1}^{m_1}\pi_{jk}^2 - \left (\sum_{j=1}^{m_1}\pi_{jk}^2 \right )^2\right ] \leq |\Omega|(|\Omega| - 1)  \pi_{\cdot k}^{(2)},
\end{align*}
where we have used in the second line that $\langle X_{i_1}(\cdot,k),X_{i_2}'(\cdot,k)\rangle^2 = \langle X_{i_1}(\cdot,k),X_{i_2}'(\cdot,k)\rangle$ since $\langle X_{i_1}(\cdot,k),X_{i_2}'(\cdot,k)\rangle$ takes values in $\{0,1\}$.

We now derive a bound on $\mathbf D$. By Jensen's inequality, we get
$$
\sum_{i} \sqrt{\mathbb E \left [f_{i}^2(X_i(\cdot,k))\right ]} \leq |\Omega|^{1/2} \sqrt{\mathbb E \left [\sum_{i}f_{i}^2(X_i(\cdot,k))\right ]} \leq |\Omega|^{1/2}
$$
where we used the bound $\mathbb E \left [\sum_{i}f_{i}^2(X_i(\cdot,k))\right ]\leq 1$.
Thus,  the Cauchy-Schwarz inequality implies
\begin{align*}
\mathbf D &\leq \sum_{i_1\neq i_2} \mathbb E [h^2(X_{i_1},X_{i_2}')]  \mathbb E^{1/2} [f^{2}_{i_1}(X_{i_1}(\cdot,k))]  \mathbb E^{1/2} [g^{2}_{i_2}(X_{i_2}'(\cdot,k))]  \\
&\leq  \max_{i_1\neq i_2} \left\{\mathbb E^{1/2} [h^2(X_{i_1},X_{i_2}')]\right\} \sum_{i_1,i_2}  \mathbb E^{1/2} [f^{2}_{i_1}(X_{i_1}(\cdot,k))]  \mathbb E^{1/2} [g^{2}_{i_2}(X_{i_2}'(\cdot,k))] \\
&\leq \max_{i_1\neq i_2} \left\{\mathbb E^{1/2} [h^2(X_{i_1},X_{i_2}')]\right\}  |\Omega| \\
&\leq  |\Omega|\left( \sum_{j=1}^{m_1}\pi_{jk}^2\right)^{1/2} = |\Omega|\left[ \pi_{\cdot k}^{(2)}\right]^{1/2},
\end{align*}
where we have used the fact that $\mathbb E [h^2(X_{i_1},X_{i_2}')] \leq \sum_{j=1}^{m_1}\pi_{jk}^2$ following from an argument similar to that used to bound $\mathbf C$.

Finally, we get a bound on $\mathbf B$. Set $\pi_{0,k} = 1 - \pi_{\cdot,k}$. Note first that 
\begin{align*}
\left\| \sum_{i_1} \mathbb E h^2(X_{i_1}(\cdot,k),\cdot)  \right\|_{L^\infty} &= |\Omega| \max_{0\leq j'\leq m_1} \left\{ \sum_{j=0}^{m_1} h^{2}(e_j(m_1),e_{j'}(m_1))\pi_{jk}  \right\}\\
&\leq |\Omega| (\pi_{\cdot k}^{(2)} )^2  + |\Omega|\max_{1\leq j' \leq m_1}\pi_{j',k}.
\end{align*}
By symmetry, we obtain the same bound on $ \left\| \sum_{i_2} \mathbb E h^2(\cdot,X_{i_2}(\cdot,k))  \right\|_{L^\infty} $. Thus we have
\begin{align*}
\mathbf B&\leq |\Omega|^{1/2}\left(\pi_{\cdot k}^{(2)}  + \max_{1\leq j' \leq m_1}\pi^{1/2}_{j',k}\right).
\end{align*}

Set now
$
U_2 = \sum_{i_1\neq i_2}h(X_{i_1}(\cdot,k), X_{i_2}(\cdot,k)).
$
We apply a decoupling argument (See for instance Theorem 3.4.1 page 125 in \cite{GinePena}) to get  that there exists a constant $C>0$, such that for any $u>0$
$$
\mathbb P\left( \sum_{i_1\neq i_2}h\left (X_{i_1}(\cdot,k), X_{i_2}(\cdot,k)\right ) \geq u \right)  \leq C \mathbb P\left(  \sum_{i_1\neq i_2}h(X_{i_1}(\cdot,k), X_{i_2}^{'}(\cdot,k)) \geq u/C\right),
$$
where $X_i'(\cdot,k)$  is independent of $X_i(\cdot,k)$ and has the same distribution as $X_i(\cdot,k)$. 
Next, Theorem 3.3 in \cite{GineLatalaZinn} gives that, for any $u>0$,
$$
 \mathbb P\left(  \sum_{i_1\neq i_2}h(X_{i_1}(\cdot,k), X_{i_2}^{'}(\cdot,k)) \geq u\right) \leq  C \exp\left[ - \frac{1}{C}\min\left( \frac{u^2}{\mathbf C^2},\frac{u}{\mathbf D},\frac{u^{2/3}}{\mathbf B^{2/3}}, \frac{u^{1/2}}{\mathbf A^{1/2}}\right) \right],
$$
for some absolute constant $C>0$. 
Combining the last display with our bounds on $\mathbf A,\mathbf B,\mathbf C,\mathbf D$, we get that for any $t>0$, with probability at least $1- 2e^{-t}$, 
\begin{align*}
\left| \frac{1}{N^2}\sum_{i_1\neq i_2}\sum_{j= 1}^{m_1} X_{i_1}(j,k)X_{i_2}(j,k)\right| &\leq  \frac{|\Omega|(|\Omega| - 1)}{N^2}\pi_{\cdot k}^{(2)}
 +  \frac{C}{N^2} \left( \mathbf C t^{1/2}+ \mathbf D t + \mathbf B t^{3/2} + \mathbf A t^2 \right)\\
 &\leq \frac{|\Omega|(|\Omega| - 1)}{N^2}\pi_{\cdot k}^{(2)} + C \left[ \frac{|\Omega|(|\Omega| - 1)}{N^2}\pi_{\cdot k}^{(2)} t^{1/2}\right. \\
 &\left. \hskip -1.5 cm+  \frac{|\Omega|}{N^2}\left( \pi_{\cdot k}^{(2)} \right)^{1/2} t +    \frac{|\Omega|^{1/2}}{N^2}\left(\pi_{\cdot k}^{(2)}  + \max_{1\leq j' \leq m_1}\pi^{1/2}_{j',k}\right) t^{3/2} +\frac{ t^2}{N^2} \right],
\end{align*}
where $C>0$ is a numerical constant.
Combining the last display with (\ref{W-diag term}) we get that, for any $t>0$ with probability at least $1-3e^{-t}$,
\begin{align*}
\sum_{j=1}^{m_1}\left(\frac{1}{N} \sum_{i \in \Omega} X_i(j,k)\right)^2  &\leq \frac{|\Omega|(|\Omega| - 1)}{N^2}\pi_{\cdot k}^{(2)}  +  C \left[ \left(\frac{|\Omega|(|\Omega| - 1)}{N^2}\pi_{\cdot k}^{(2)}+ \frac{2\sqrt{|\Omega| \pi_{\cdot k}}}{N^2}\right) t^{1/2}\right. \\
 &\left.\hskip -2 cm + \frac{|\Omega|}{N^2} \pi_{\cdot k}+  \left(\frac{|\Omega|}{N^2}\left(\pi_{\cdot k}^{(2)} \right)^{1/2} + \frac{1}{N^2}\right)t +    \frac{|\Omega|^{1/2}}{N^2}\left(\pi_{\cdot k}^{(2)}  + \max_{1\leq j' \leq m_1}\pi^{1/2}_{j',k}\right) t^{3/2} +\frac{ t^2}{N^2} \right].
\end{align*}
Set $\pi_{\max} = \max_{1\leq k \leq m_2}\{\pi_{\cdot k}\}$ and $\pi^{(2)}_{\max} = \max_{1\leq k \leq m_2}\{\pi_{\cdot k}^{(2)}\}$. Using the union bound and up to a rescaling of the constants, we get that, with probability at least $1 - e^{t}$,
\begin{align*}
\|W\|_{2,\infty}^2 &\leq \frac{|\Omega|(|\Omega| - 1)}{N^2}\pi^{(2)}_{\max}   + C \left[ \left(\frac{|\Omega|(|\Omega| - 1)}{N^2}\pi^{(2)}_{\max} + \frac{2\sqrt{|\Omega| \pi_{\max}}}{N^2}  \right) (t+\log m_2)^{1/2}\right. \\
 &\left.\hskip 1.5 cm +\frac{|\Omega|}{N^2} \pi_{\max}+  \frac{|\Omega|}{N^2}\left(\pi^{(2)}_{\max} \right)^{1/2} (t+\log m_2)  \right .\\&\left .\hskip 2 cm+   \frac{|\Omega|^{1/2}}{N^2}\left(\pi^{(2)}_{\max}   + \max_{j,k}\{\pi_{jk}^{1/2}\}\right) (t+\log m_2)^{3/2} +\frac{ (t+\log m_2)^2}{N^2} \right].
\end{align*}
Recall that $\vert \Omega\vert=n$ and $\ae=N/n$. Assumption \ref{marginal_columns} and the fact that $n\leq \vert \mI\vert$ imply  that there exists a numerical constant $C>0$ such that, with probability at least $1 -e^{-t}$, 
\begin{align*}
\|W\|_{2,\infty}^2 &\leq C \left(\frac{\gamma^{2} }{\ae N m_2}\left (\sqrt{t+\log m_2}+(t+\log m_2)\sqrt{\frac{m_2}{n}}\right ) + \frac{(t+\log m_2)^{2}}{N^{2}}   \right) 
\end{align*}
where we have used that $\pi_{j,k} \leq \pi_{\cdot k} \leq \sqrt{2}\gamma/m_2$.
Finally, the bound on the expectation $\mathbb E \|W\|_{2,\infty}$ follows from this result and Lemma \ref{lem-tech}. 


\section{Proof of Lemma \ref{lemma_stochastique-sparse}}

With the notation $X_i(j,k) = \langle X_i, e_j(m_1)e_k(m_2)^{\top}\rangle$ we have
$$
\|\Sigma\|_{\infty} = \max_{1\leq j\leq m_1,1\leq k\leq m_2} \left| \frac{1}{N}\sum_{i \in \Omega} \xi_i X_i(j,k) \right|.
$$
Under Assumption \ref{noise}, the Orlicz norm $\|\xi_i\|_{\psi_2}= \inf\{x>0: \mathbb E[(\xi_i/x)^2] \le e\}$ satisfies $\|\xi_i\|_{\psi_2} \leq c\sigma$ for some numerical constant $c>0$ and all $i$.  This and the relation (See Lemma 5.5 in \cite{vershynin}\footnote{this statement actually appears as an intermediate step in the proof of this lemma.})
$$
\mathbb E[|\xi_i|^\ell] \leq \frac{\ell}{2} \Gamma\left(\frac{\ell}{2}\right) \|\xi_i\|_{\psi_2}^\ell,\quad \forall \ell\geq 1,
$$
imply that $N^{-\ell}\mathbb E[|\xi_i|^\ell X_i^\ell(j,k)] = N^{-\ell} \mathbb E[X_i(j,k)] \mathbb E[|\xi_i|^\ell]\le (\ell !/2) c^2v(c\sigma/N)^{\ell-2}$ for all $\ell\geq 2$ and $v=\frac{\sigma^2\mu_1}{N^2 m_1m_2}$, where we have used the independence between $\xi_i$ and $X_i$, and Assumption \ref{marginal_sparse}.
Thus, for any fixed $(j,k)$, we have 
$$
\sum_{i\in \Omega}\mathbb E \left[  \frac{1}{N^2} \xi_i^2  X_i^2(j,k) \right] \leq |\Omega|\frac{c^2\sigma^2\mu_1}{N^2 m_1m_2}=\frac{c^2\mu_1  \sigma^2}{\ae N m_1m_2} =:v_1,
$$
and
$$
\sum_{i\in \Omega}\mathbb E \left[  \frac{1}{N^\ell} |\xi_i|^\ell  X_i^\ell(j,k) \right] \leq \frac{\ell!}{2}v_1\left(\frac{c\sigma}{ N}\right)^{\ell-2}.
$$
Thus, we can apply Bernstein's inequality (see, e.g. \cite{van_de_geer}, page 486), which yields
$$
\mathbb P \left( \left| \frac{1}{N}\sum_{i \in \Omega}\xi_i X_i(j,k)  \right| > C\left(\sqrt{ \frac{\mu_1  \sigma^2 t}{\ae N m_1m_2}} + \frac{\sigma t}{N}\right)\right) \leq 2e^{-t}
$$
for any fixed $(j,k)$. Replacing here $t$ by $t+\log(m_1m_2)$ and using the union bound we obtain
$$
\mathbb P \left( \left\| \Sigma \right\|_{\infty} > C\left(\sqrt{ \frac{\mu_1  (t+\log(m_1m_2))}{\ae N m_1m_2}} + \frac{(t+\log(m_1m_2))}{N}\right)\right) \leq 2e^{-t}.
$$
The bound on $\mathbb E[\|\Sigma\|_{\infty}]$ in the statement of Lemma~\ref{lemma_stochastique-sparse} follows from this inequality and Lemma~\ref{lem-tech}.
The same argument proves the bounds on $\|\Sigma_R\|_{\infty}$ and $\mathbb E  \|\Sigma_R\|_{\infty}$ in the statement of Lemma~\ref{lemma_stochastique-sparse}.
By a similar (and even somewhat simpler) argument, we also get that
$$
\mathbb P \left( \left\| W - \mathbb E[ W] \right\|_{\infty} >  C\left(\sqrt{ \frac{\mu_1   (t+\log(m_1m_2))}{\ae N m_1m_2}} + \frac{t+\log(m_1m_2)}{ N}\right) \right) \leq 2e^{-t}
$$
while Assumption \ref{marginal_sparse} implies that $\|\mathbb E[W]\|_{\infty} \leq \frac{\mu_1}{\ae m_1m_2}$.

\section{Technical Lemmas}

\begin{lemma}\label{lem-tech}
Let $Y$ be a non-negative random variable. Let there exist $A\geq 0$, and $a_j>0$, $\alpha_j>0$ for $1\leq j\leq m$,  such that
$$
\mathbb P(Y> A + \sum_{j=1}^m a_j t^{\alpha_j}) \leq e^{-t},\quad \forall t>0.
$$
Then 
$$
\mathbb E [Y] \leq A  + \sum_{j=1}^m a_j \alpha_j \Gamma(\alpha_j),
$$
where $\Gamma(\cdot)$ is the Gamma function.
\end{lemma}

\begin{proof}
Using the change of variable
$
u = \sum_{j=1}^m a_j v^{\alpha_j}
$
we get
\begin{align*}
\mathbb E [Y] &= \int_{0}^{\infty}\mathbb P(Y>t)dt 
\leq A +  \int_{0}^{\infty}\mathbb P(Y>A + u)du\\
 &= A+ \int_{0}^{\infty}\mathbb P(Y>A + \sum_{j=1}^m a_j v^{\alpha_j})\left( \sum_{j=1}^m a_j \alpha_j v^{\alpha_j-1}\right) dv\\
&\leq A+ \int_{0}^{\infty}\left( \sum_{j=1}^m a_j \alpha_j v^{\alpha_j-1}\right)e^{-v} dv = A+ \sum_{j=1}^m a_j \alpha_j \Gamma(\alpha_j).
\end{align*}
\end{proof}

\begin{lemma}\label{lem-44}
Assume that $\mR$ is an absolute norm. Then
\begin{equation*}
\mR^{*}\left (\frac{1}{N}\sum_{i\in \Omega} \left\langle X_i,\Delta L\right\rangle X_i\right )\leq2\mathbf{a}\mR^{*}(W)
\end{equation*}
where $W=\frac{1}{N}\sum_{i\in \Omega}X_{i}$.
\end{lemma}

\begin{proof}
In view of the definition of $\mR^{*}$,
\begin{equation*}\label{robust_4}
\begin{split}
\frac1{2\mathbf{a}}\mR^{*}\left (\dfrac{1}{N}\sum_{i\in \Omega} \left\langle X_i,\Delta L\right\rangle X_i\right )&=\underset{\mR(B)\leq 1}{\sup}\left \langle\dfrac{1}{N}\sum_{i\in \Omega} \frac{\left\langle X_i,\Delta L\right\rangle}{2\mathbf{a}}X_i,B\right \rangle\\
&\le \underset{\mR(B')\leq 1} {\sup}\left \langle\dfrac{1}{N}\sum_{i\in \Omega}  X_i,B'\right \rangle = \mR^{*}(W),
\end{split}
\end{equation*}
where we have used the inequalities $\left\langle X_i,\Delta L\right\rangle \le \Vert\Delta L\Vert_{\infty}\leq 2\mathbf{a}$, and the fact that $\mR$ is an absolute norm. 
\end{proof}

\section*{Acknowledgement}
The work of O. Klopp was conducted as part of the project Labex MME-DII (ANR11-LBX-0023-01). The work of K. Lounici was supported in part by Simons Grant 315477 and by NSF Career Grant DMS-1454515.
The work of A.B.Tsybakov was supported by GENES and by the French National Research Agency (ANR) under the grants 
IPANEMA (ANR-13-BSH1-0004-02), Labex ECODEC (ANR - 11-LABEX-0047), ANR -11- IDEX-0003-02, and by the "Chaire Economie et Gestion des Nouvelles Donn\'ees", under the auspices of Institut Louis Bachelier, Havas-Media and Paris-Dauphine.
\\

%


\end{document}